\documentclass[final]{amsart} 
\usepackage{amsmath} 
\usepackage{amssymb} 
\usepackage{amsthm}

\newcommand{\norm}[1]{\left\|{#1}\right\|}
\newcommand{\abs}[1]{\left\lvert{#1}\right\rvert} 
 
\newcommand{\ang}[1]{\left\langle{#1}\right\rangle}
\newcommand{\N}{\mathbb N}

\newcommand{\R}{\mathbb R}
\newcommand{\Z}{\mathbb Z}

\newcommand{\T}{\mathbb T}
\newcommand{\ds}{\displaystyle}
\newcommand{\FF}{\mathcal F}
\newcommand{\n}{\nabla}
\newcommand{\p}{\partial}
\newcommand{\al}{\alpha}
\newcommand{\be}{\beta}
\newcommand{\ga}{\gamma}
\newcommand{\de}{\delta}
\newcommand{\ep}{\varepsilon}
\newcommand{\io}{\iota}
\newcommand{\ka}{\kappa}
\newcommand{\ti}{\widetilde}

\newtheorem{thm}{Theorem}[section]

\newtheorem{lem}[thm]{Lemma}
\newtheorem{prop}[thm]{Proposition}
\theoremstyle{remark}
\newtheorem{rem}[thm]{Remark}

\theoremstyle{definition}

\numberwithin{equation}{section}

\begin{document}

\title{The Cauchy problem for wave maps on a curved background}
\author{Andrew  Lawrie}
\address{Department of  Mathematics, The University of Chicago\\ Chicago, IL 60637, U.S.A.} 
\email{alawrie@math.uchicago.edu} 
\date{\today}
\thanks{This work is part of the author's Ph.~ D.~ thesis at the University of Chicago}

\begin{abstract}
We consider the Cauchy problem for wave maps $u: \R\times M \to N$, for Riemannian manifolds $(M, g)$ and $(N,h)$. We prove global existence and uniqueness for initial data, $u[0]=(u_0, u_1)$, that is small in the critical norm $\dot{H}^{\frac{d}{2}}\times \dot{H}^{\frac{d}{2}-1}(M; TN)$, in the case $(M,g)$ = $(\R^4, g)$, where $g$ is a small perturbation of the Euclidean metric.   The proof follows the method introduced by Statah and Struwe in ~\cite{Shat-Stru WM}  for proving global existence and uniqueness of small data wave maps $u : \R\times \R^d \to N$ in the critical norm, for $d\ge 4$. In our  argument we employ the Strichartz estimates for variable coefficient wave equations established by Metcalfe and Tataru in ~\cite{Met-Tat}. 
\end{abstract}

\maketitle

\newcounter{parts}

\vspace{\baselineskip}

\section{Introduction}

\vspace{\baselineskip}

\subsection{Wave Maps}\label{Wave Maps}  Wave maps are the hyperbolic analogs of harmonic maps. They are defined as follows. Let $(M,g)$ be a Riemannian manifold of dimension $d$. Denote by $(\ti{M}, \eta)$ the Lorentzian manifold $\ti{M} = \R\times M$. The metric $\eta$ is represented in local coordinates by $\eta = (\eta_{\al \be}) = \textrm{diag} (-1,g_{ij})$. Let $(N,h)$ be a complete Riemannian manifold without boundary of dimension $n$.
\subsubsection{Intrinsic Definition}\label{Intrinsic Definition}
  A map $u: (\ti{M},\eta) \longrightarrow (N,h)$ is called a wave map if it is, formally,  a critical point of the functional 

\begin{align*}
\mathcal{L}(u) = \frac{1}{2} \int_{\ti{M}} \ang{du, du}_{T^*\ti{M} \otimes u^*TN} \, \textrm{dvol}_{\eta}
\end{align*}
\\
Here we view the differential, $du$, of the map $u$ as a section of the vector bundle $(T^*\ti{M} \otimes u^*TN, \eta\otimes u^*h)$, where $u^*TN $ is the pullback of $TN$ by $u$ and $u^*h$ is the pullback metric. In local coordinates this becomes

\begin{align*}
\mathcal{L}(u) 
=\frac{1}{2} \int_{\ti{M}} \eta^{\al\be}(z) h_{ij}(u(z))\p_{\al}u^i(z) \p_{\be} u^j(z) \, \sqrt{\abs{\eta}} \,dz\\
\end{align*}
One can show that the Euler-Lagrange equations for $\mathcal{L}$ are given by

\begin{align}\label{wm}
\frac{1}{\sqrt{\abs{\eta}}}D_{\al}\left( \sqrt{\abs{\eta}} \eta^{\al \be} \p_{\be}u\right) =0
\end{align}
\\
where $D$ is the pull-back covariant derivative on $u^*TN$. In local coordinates on $N$, writing $u=(u^1, \dots, u^n)$,  we can rewrite ~\eqref{wm} as 

\begin{align}\label{wm local}
\Box_{\eta} u^{k} = -\eta^{\al \be} \Gamma_{ij}^{k}(u) \p_{\al} u^i \p_{\be} u^j
\end{align}
where $\Box_{\eta}u = -\p_{tt}u + \Delta_gu$  and $\ds{\Delta_gu= \frac{1}{\sqrt{\abs{g}}} \p_{\al}( \sqrt{\abs{g}}g^{\al \be} \p_{\be}u)} $ is the Laplace-Beltami operator on $M$.   $\Gamma_{ij}^k= \frac{1}{2} h^{k\ell}( \p_i h_{\ell j} + \p_j h_{i\ell} - \p_{\ell} h_{ij})$ are the Christoffel symbols associated to the metric connection on ~$N$.   

\subsubsection{Extrinsic Definition}\label{Extrinsic Definition} Wave maps can also be defined extrinsically. This approach is equivalent to the intrinsic approach, see for example ~\cite[Chapter 1]{Sha-Stru GWE}.  By the Nash-Moser embedding theorem there exists $m\in\N$  large enough so that we can isometrically embed $(N,h)\hookrightarrow (\R^m, \ang{\cdot, \cdot})$, where  $\ang{\cdot, \cdot}$ is the Euclidean scalar product. We can thus consider maps $u: (\ti{M}, \eta) \to (\R^m, \ang{\cdot, \cdot})$ such that $u(t,x) \in N$ for every $(t,x)\in \ti{M}$.  Wave maps can then be defined formally as critical points of the functional 

\begin{align*}
\mathcal{L}(u) =\frac{1}{2} \int_{\ti{M}} \eta^{\al \be} \ang{\p_{\al}u, \p_{\be}u} \, \sqrt{\abs{\eta}} \, dz
\end{align*}
One can show that $u$ is a wave map if and only if  $u$ satisfies 

\begin{align}\label{wm extrinsic}
\Box_\eta u \perp T_{u}N  
\end{align}
From this we can deduce that $u$ satisfies 

\begin{align} \label{wm extrinsic local}
\Box_{\eta} u = -\eta^{\al \be}  S(u)(\p_{\al}u, \p_{\be}u)
\end{align}
where $S$ is the second fundamental form of the embedding $N\hookrightarrow \R^m$.  One can formally establish energy conservation from the  extrinsic definition ~\eqref{wm extrinsic}. Define the energy

\begin{align}\label{energy}
E(u(t)) := \frac{1}{2}\int_M \left( \abs{\p_t u}^2 + \abs{d_M u}^2\right) \sqrt{\abs{g}} \, dx
\end{align}
where by $d_Mu $ we mean the differential of the map $u(t) : M\to \R^m$. Observe that $\ds{\Box_{\eta}u \perp T_{u}N}$ implies that $\ang{\Box_{\eta}u, \p_t u} =0$. Hence we have 
 
\begin{align*}
0 &= -\int_{M} \ang{\Box_{\eta}u, \p_t u}_{u(x)} \sqrt{\abs{g}} \, dx\\
 &= \int_{M} \ang{\p_t\p_tu, \p_t u}_{u(x)}\sqrt{\abs{g}}\,dx -\int_M \ang{ \p_{\al}( \sqrt{\abs{g}}g^{\al \be} \p_{\be}u), \p_t u}_{u(x)}  \, dx\\
&= \frac{1}{2}\int_M \frac{d}{dt} \abs{\p_t u}^2 \sqrt{\abs{g}} \, dx +  \int_M \ang{g^{\al \be} \p_{\be} u, \p_{\al} \p_{t} u }_{u(x)}  \sqrt{\abs{g}} \, dx\\
&= \frac{d}{dt}\left(\frac{1}{2}\int_M  \left(\abs{\p_t u}^2 + \abs{d_M u}^2\right) \sqrt{\abs{g}} \, dx\right)
\end{align*}
Integrating in time then gives $\ds{E(u(t)) = E(u(0))}$ for any time $t$.

\vspace{\baselineskip}

\subsection{History and Overview}\label{History and Overview} The wave maps equation has been studied extensively over the past several decades in the case of a flat background manifold, $(M,g) = (\R^d, \ang{ \cdot, \cdot})$. In this case, $\eta$ is the Minkowksi metric on $\R^{1+d}$ and the intrinsic formulation ~\eqref{wm} becomes 

\begin{align}\label{flat wm}
D_{\al} \p^{\al} u = 0
\end{align}
The extrinsic formulation ~\eqref{wm extrinsic}  is given by 

\begin{align}\label{flat ex wm}
\Box u \perp T_uN
\end{align}
In this setup, wave maps are invariant under the scaling~ $u(t,x) \mapsto u_{\lambda}(t,x) = u(\lambda t, \lambda x)$.  This scaling is critical relative to $\dot{H}^{\frac{d}{2}}\times \dot{H}^{\frac{d}{2}-1}(\R^d)$ whereas the conserved energy, $E(u)$, is critical relative to $\dot{H}^1\times L^2(\R^d)$. Hence the Cauchy problem for ~\eqref{flat ex wm} is critical in $\dot{H}^{\frac{d}{2}}\times \dot{H}^{\frac{d}{2}-1}(\R^d)$ and energy critical when $d=2$. 

We review some of the major developments in the theory of wave maps. In the energy super-critical case, $d\ge 3$, Shatah in ~\cite{Sha 1} showed that self-similar blow-up can occur for solutions of finite energy. In the energy critical case, $d=2$, there is no self similar blow-up as demonstrated by Shatah and Struwe in ~\cite{Sha-Stru GWE}.  In the equivariant setting, Struwe proved in  ~\cite{Stru 1} that if blow-up does occur then the solution must converge, after rescaling, to a non-constant, co-rotational harmonic map. Recently, Krieger, Schlag, and Tataru in ~\cite{Kri-Sch-Tat} and Rodnianski-Sterbenz in ~\cite{Rod-Ster} have constructed finite energy wave maps $u: \R^{1+2} \to S^2$ that blow up in finite time.

The well-posedness theory for energy critical wave maps in the equivariant setting was developed by Christodoulou and Tahvildar-Zadeh in  ~\cite{Chr-Tah 1}, ~\cite{Chr-Tah 2}, and by Shatah and  Tahvildar-Zadeh in ~\cite{Sha-Tah 1} ~\cite{Sha-Tah 2}.  In the non-equivariant case, Klainerman and Machedon in  ~\cite{Kla-Mac 1}, ~\cite{Kla-Mac 2}, ~\cite{Kla-Mac 3}, ~\cite{Kla-Mac 4}, and Klainerman and Selberg in ~\cite{Kla-Sel 1}, ~\cite{Kla-Sel 2},  established strong well-posedness in the subcritical norm $H^s\times H^{s-1}(\R^d)$ with $s> \frac{d}{2}$ by exploiting the null-form structure present in ~\eqref{flat ex wm}. 

The first major breakthrough in the critical theory, $s=\frac{d}{2}$, was accomplished by Tataru in ~\cite{Tat 1}, ~\cite{Tat 2}, where he proved   global well-posedness for smooth data that is small in the Besov space $\dot{B}^{\frac{d}{2}}_{2,1}\times \dot{B}^{\frac{d}{2}-1}_{2,1}(\R^d)$ for $d\ge 2$. Then, in the groundbreaking work ~\cite{Tao 1}, ~\cite{Tao 2}, Tao proved  global well-posedness  for wave maps $u:\R^{1+d} \to S^k$ for smooth data that is small in the critical Sobolev norm $\dot{H}^{\frac{d}{2}}\times \dot{H}^{\frac{d}{2}-1}(\R^d)$ for $d\ge 2$. Later, this result was extended to more general targets by Klainerman and Rodnianski in ~\cite{Kla-Rod}, by Krieger in ~\cite{Kri 1}, ~\cite{Kri 2}, ~\cite{Kri 3}, by Nahmod, Stefanov and Uhlenbeck in ~\cite{Nah-Ste-Uhl}, by Shatah and Struwe in ~\cite{Shat-Stru WM}, and by Tataru in ~\cite{Tat  3}, ~\cite{Tat 4}. 

Finally, the large data, energy critical case has been undertaken in major works by Krieger and Schlag in ~\cite{Kri-Sch},  Sterbenz and Tataru in ~\cite{Ste-Tat 1}, ~\cite{Ste-Tat 2}, and Tao in ~\cite{Tao 3}--\cite{Tao 7}. 

\vspace{\baselineskip}
The work of Shatah and Struwe in ~\cite{Shat-Stru WM} constituted a significant simplification of Tao's argument in dimensions $d\ge 4$, and it is on the methods utilized in ~\cite{Shat-Stru WM}, that this present work is based. In ~\cite{Shat-Stru WM}, Shatah and Struwe consider the Cauchy problem for wave maps $u:\R^{1+d} \to N$ with initial data $(u_0, u_1)\in H^{\frac{d}{2}} \times H^{\frac{d}{2}-1} (\R^d, TN)$ that is small in the critical norm $\dot{H}^{\frac{d}{2}} \times \dot{H}^{\frac{d}{2}-1}(\R^d, TN)$ for $d\ge 4$. The target manifold $N$ is assumed to have bounded geometry. Their main result is a proof of the existence of a unique global solution, $(u, \dot{u}) \in C^0( \R; H^{\frac{d}{2}}) \times C^0(\R; H^{\frac{d}{2}-1})$, to the aforementioned Cauchy problem. Existence is deduced by way of the following global a priori estimates for the differential, $du$, of the wave map:

\begin{align*}
\|du\|_{L^{\infty}_t \dot{H}^{\frac{d}{2}-1}_x} + \|du\|_{L^2_tL^{2d}_x} \lesssim \|u_0\|_{\dot{H}^{\frac{d}{2}}} + \|u_1\|_{\dot{H}^{\frac{d}{2}-1}}
\end{align*} 

In order to prove  the above estimates, the Coulomb frame is introduced as this allows one to derive a system of wave equations for $du$ that is amenable to a Lorentz space version of the endpoint Strichartz estimates proved in ~\cite{Kee-Tao}. The connection form, $A$, associated to the Coulomb frame on the vector bundle $u^*TN$ appears in the nonlinearity  of the wave equation for $du$, and estimates to control its size are crucial to the argument. The Coulomb gauge condition implies that $A$ satisfies a certain elliptic equation, and it is this structure that enables the proof, for example, of the essential $L^1_t L^{\infty}_x$ estimates for $A$, see ~\cite[Proposition $4.1$]{Shat-Stru WM}.  

\vspace{\baselineskip}

In this paper, we consider the Cauchy problem for wave maps $u: \R\times M \to N$, where the background manifold $(M,g)$ is no longer Euclidean space. We follow the same basic argument as in ~\cite{Shat-Stru WM} and derive a wave equation for the $u^*TN$-valued $1$-form, $du$, using the Coulomb gauge as our choice of frame on $u^*TN$. As the geometry of $(M,g)$ is no longer trivial, the resulting equation for $du$ is, in its most natural setting, an equation of $1$-forms. In coordinates on $M$, we can rewrite the equation for $du$ in components, obtaining a system of variable coefficient nonlinear wave equations. This is the content of Section ~\ref{wave equation for du}. 

The main technical ingredients in ~\cite{Shat-Stru WM} are the estimates for the connection form $A$,  and the endpoint Strichartz estimates for the wave equation used to control the $L^{\infty}_t\dot{H}_x^{\frac{d}{2}-1}\cap L^2_tL^{2d}_x$ norm of $du$. In order to proceed as in ~\cite{Shat-Stru WM}, but now in the setting of a curved background manifold, we will need replacements for each of these items.  

In what follows, we restrict our attention to the case that the background manifold $(M,g)$ is $(\R^4, g)$, with $g$  a small perturbation of the Euclidean metric, as in this case we have suitable replacements for the technical tools used in ~\cite{Shat-Stru WM}. Here we view the equations for the components of connection form, $A$,  as a system of variable coefficient elliptic equations and prove elliptic estimates via a perturbative argument, see ~Proposition ~\ref{main elliptic estimates}.  We employ  several tools from the theory of Lorentz spaces to prove the crucial $L^1_t L^{\infty}_x$ estimates for $A$. 

In order to have suitable Strichartz estimates, we tailor our assumptions on the metric $g$ so that the variable coefficient wave equations for $du$ are of the type studied by Metcalfe and Tataru in ~\cite{Met-Tat}. We deduce a Lorentz refinement to the Strichartz estimates in ~\cite{Met-Tat}, see Section ~\ref{Strichartz Estimates} below, which we use to prove global a priori estimates for $du$ in Section ~\ref{a priori estimates}.

 The global-in-time Strichartz estimates for variable coefficient wave equations in ~\cite{Met-Tat} that we use in the proof of the a priori estimates for $du$ have emerged from Tataru's method of  using phase space transforms and microlocal analysis to prove dispersive estimates for variable coefficient dispersive equations. In the case of the variable coefficient wave equation, the Bargmann transform is used to construct a parametrix that satisfies suitable dispersive estimates. Localized energy estimates are then used to control error terms when proving estimates for the variable coefficient operator.  We refer the reader to ~\cite{Tat 5}--\cite{Tat 10} and of course to ~\cite{Met-Tat}, for more details and history.  A very brief summary is included in Section ~\ref{Strichartz Estimates}. 

Our main theorem is a global existence and uniqueness result for the Cauchy problem for wave maps in this setting, with data $(u_0, u_1)$ that is small in the critical norm $\dot{H}^{\frac{d}{2}} \times \dot{H}^{\frac{d}{2}-1}(M , TN)$. The precise statement of the result is Theorem ~\ref{main thm} below.

\vspace{\baselineskip}
\subsection{Acknowledgements}\label{Acknowledgements} I would like to thank my advisor, Professor Wilhelm Schlag, for introducing me to the topic of wave maps and for his generous help, encouragement, and guidance related to this work.

\vspace{\baselineskip}
 
 \subsection{Notation}\label{Notation} In what follows we will adopt the convention that $f \lesssim g$ means that there exists a constant $C>0$ such that $f \le C g$.  Similarly, $f \simeq g$ will mean that there exist constants $c, C >0$ such that $c g \le f \le Cg$.

\vspace{\baselineskip}

\subsection{Geometric Framework}\label{geometric framework}

We set $(M,g) = (\R^4, g)$ with $g$ a small perturbation of the Euclidean metric on $\R^4$, satisfying the following assumptions: Let $\ep>0$ be a small constant, to be specified later. We will require 

\begin{align}
\|g -g_0\|_{L^{\infty}} &\le \ep \label{g}\\
\| \p g \|_{L^{4,1}(\R^4)} &\lesssim \ep \label{p g} \\
\| \p^2 g\|_{L^{2,1}(\R^4)} &\lesssim \ep \label{p^2 g}\\
\| \p^k g\|_{L^2(\R^4)} &< \infty \quad \mathrm{for}\, \, k \ge 3 \label{p^k g}
\end{align}
where $g_0= \mathrm{diag}(1, 1, 1, 1)$ is the Euclidean metric on $\R^4$ and $L^{p,q}(\R^4)$ denotes the Lorentz space. Assumptions ~\eqref{g}--\eqref{p^2 g} are needed in order to prove the elliptic estimates for the connection form, $A$, associated to the Coulomb frame in Section ~\ref{connection form estimates}. Note that these assumptions are consistent with, and are, in fact, stronger than the weak asymptotic flatness conditions specified in Metcalfe--Tataru in ~\cite{Met-Tat}, namely 

\begin{align}\label{scale inv metric}
\sum_{j\in \Z} \sup_{\abs{x} \simeq 2^j} \abs{x}^2 \abs{\p^2 g(x) } + \abs{x} \abs{\p g(x)} + \abs{ g(x) - g_0} \le \ep
\end{align}
This will justify our application in Section ~\ref{a priori estimates} of the Strichartz estimates for variable coefficient wave equations deduced in ~\cite{Met-Tat}. 

The assumptions in \eqref{p^k g} are needed in order to establish the high regularity local theory for wave maps. This theory will be used in the existence argument in Section \ref{Existence}.

We will also record a few comments regarding the assumptions on the target manifold $(N,h)$. We will assume that $(N,h)$ is a smooth complete Riemannian manifold, without boundary that is isometrically embedded into $\R^m$. Following ~\cite{Shat-Stru WM}, we also assume that $N$ has bounded geometry in the sense that the curvature tensor, $R$, and the second fundamental form, $S$, of the embedding are bounded and all of their derivatives are bounded.  

In the argument that follows, we will assume that either  $N$ admits a parallelizable structure or that $N$ is compact,  as we will require a global orthonormal frame for $TN$ in our argument. Such a frame does not, of course, exist for a general compact manifold. However if $N$ is compact, by an argument in ~\cite{Hel}, we can avoid this inconvenience by constructing a certain isometric embedding $J: N \hookrightarrow \ti{N}$ where $\ti{N}$ is diffeomorphic to the flat torus $\T^m$ and admits an orthonormal frame. This embedding $J$ is constructed so that $u$ is a wave map if and only if the composition $J\circ u$ is a wave map, see ~\cite[Lemma $4.1.2$]{Hel}. This allows us to work with $J\circ u: \ti{M} \to \ti{N}$ instead of with $u$. Hence we can assume without loss of generality that the target manifold $N$ admits a global orthonormal frame $\ti{e} = (\ti{e}_1, \dots, \ti{e}_n)$ for the tangent space $TN$.

\subsection{Main Result}\label{main result}

In this paper, we study the Cauchy problem for wave maps. The initial data, $(u, \dot{u})\vert_{t=0}=(u_0, u_1)$, can either be viewed intrinsically or extrinsically.  In the extrinsic formulation, we will consider initial data 
\begin{align}\label{extrinsic data}
(u_0, u_1) \in (M ,g) \to TN
\end{align}
by which we mean $u_0(x) \in N \hookrightarrow \R^m$ and $u_1(x) \in T_{u_0(x)}N \hookrightarrow \R^m$ for almost every $x \in M$.  And we say that $(u_0, u_1) \in H^s_e\times H^{s-1}_e(M; TN)$ if $u_0 \in H^s(M; \R^m)$ and $u_1\in H^{s-1}(M; \R^m)$. The homogeneous spaces $ \dot{H}^s_e\times \dot{H}^{s-1}_e(M; TN)$ are defined similarly. For the definition of the spaces $H^s(M; \R^m)$ we refer the reader to Section ~\ref{sobolev spaces},   or to ~\cite{Heb}. 

To view the data intrinsically, we will put to use the parallelizable structure on $TN$. Let our initial data be given by $(u_0, u_1)$ where $u_0: M\to N$  and $u_1 :M \to u_0^*TN$ with $u_1(x) \in T_{u_0(x)}N$.  
Observe that $u_0^* TN$ inherits a parallelizable structure from $TN$, see Section ~\ref{Coulomb Frame},   and let $e=(e_1, \dots, e_n)$ be an orthonormal frame for $u^*TN$.  Since $du_0: TM \to u^*TN$ we can find a $u^*TN$-valued $1$-form $q_0=q_0^a e_a$ such that $du_0= q^a_0 e_a$. Similarly we can find $q_1^a: M\to \R$ such that $u_1 = q^a_1 e_a$. We then say that $(u_0, u_1) \in H^s_i \times H^{s-1}_i(M; TN)$  if $q_0^a \in H^{s-1}(TM; \R)$ and $q_1^a \in H^{s-1}(M; \R)$ for each $1\le a\le n$. These norms are further discussed in Section ~\ref{sobolev spaces}. Again, the homogeneous versions $ \dot{H}^s_i\times \dot{H}^{s-1}_i(M; TN)$ are defined similarly. 

In Section ~\ref{Equivalence of Norms}, we show that if we choose the frame $e$ to be the Coulomb frame, see Section ~\ref{Coulomb Frame}, then the extrinsic and intrinsic approaches to defining the homogeneous Sobolev norms of our data $(u_0, u_1)$ are equivalent. This will allow us to use both definitions interchangeably in the arguments that follow.

 Also in the appendix, Section ~\ref{sobolev spaces},  we show that the ``covariant" Sobolev spaces $\dot{H}^s(M; N)$, with $(M,g)= (\R^4, g)$ with the metric $g$ as in ~\eqref{g}--\eqref{p^k g} are equivalent to the ``flat'' spaces $\dot{H}^s((\R^4,g_0); N)$ with the Euclidean metric $g_0$ on $\R^4$. Hence in what follows we can, when convenient,  ignore the non-Euclidean metric $g$ for the purpose of estimating Sobolev norms, replacing covariant derivatives on $M$ with partial derivatives and the volume form $\textrm{dvol}_g$ with the Euclidean volume form. 

We can now state the main theorem.

\begin{thm}\label{main thm}
Let $(N, h)$ be a smooth, complete, $n$-dimensional Riemannian manifold without boundary and with bounded geometry. Let $(M,g) = (\R^4, g)$ with $g$ as in ~\eqref{g}--\eqref{p^k g} and let $(\ti{M}, \eta) = (\R\times M, \eta)$ with $\eta=\mathrm{diag}(-1, g)$. Then there exists an $\ep_0 >0$ such that for every $(u_0, u_1) \in H^2 \times H^1((M,g); TN)$ such that 

\begin{align}\label{small energy}
\|u_0\|_{\dot{H}^2} + \|u_1\|_{\dot{H}^1} < \ep_0
\end{align}
there exists a unique global wave map, $u: (\ti{M}, \eta) \rightarrow (N,h)$,  with initial data $(u, \dot{u})\vert_{t=0}=(u_0, u_1)$,   such that $(u, \dot{u}) \in  C^0(\R; H^2(M;N)) \times C^0(\R; H^1(M; TN))$. Moreover, $u$ satisfies the global estimates

\begin{align}\label{global estimates}
\|du\|_{L^{\infty}_t \dot{H}^1_x} + \|du\|_{L^2_t L^8_x} \lesssim \ep_0.
\end{align}
In addition,  any higher regularity of the data is preserved.
\end{thm}
 
 We will use a bootstrap argument to prove the global estimates ~\eqref{global estimates}. In what follows we will make the assumption that there exists a time $T$ such that for a wave map $u$ with data $(u_0, u_1)$ as in ~\eqref{small energy}, the estimates in ~\eqref{global estimates}  hold on the interval $[0, T)$. That is, we have 
 
 \begin{align} \label{bootstrap hypothesis}
 \|du\|_{L^{\infty}_t([0,T); \dot{H}_x^1)} + \|du\|_{L^2_t([0,T); L^8_x)} \lesssim \ep_0
\end{align}
We will use this assumption to prove the global-in-time estimates ~\eqref{global estimates}.
 
 \begin{rem}\label{high regularity} The local well-posedness theory for the high regularity Cauchy problem for ~\eqref{wm extrinsic local} is standard. For example, with $(M,g) = (\R^4, g)$ for a smooth perturbation $g$ as in ~\eqref{g}--\eqref{p^k g},  if we have data $(u_0, u_1) \in H^{s} \times H^{s-1}(M; TN)$ for say, $s> 4= \frac{d}{2} +2$, then the Cauchy problem for ~\eqref{wm extrinsic local} is locally well-posed. This can be proved using $H^s$ energy estimates and a contraction argument. The proof relies on the fact that $H^s(\R^d)$ is an algebra for $s>\frac{d}{2}$, and can be found for example in ~\cite{Sha-Stru GWE}.  
 \end{rem}

 \begin{rem}  We have only addressed the case $d=4$ case here because this is the only dimension where we have applicable Strichartz estimates. In dimension $3$, the endpoint $L^2_tL^{\infty}_x$ estimate is forbidden. In higher dimensions, $d\ge 4$, the initial data is assumed to be small in $\dot{H}^{s}\times \dot{H}^{s-1}$ with $s=\frac{d}{2}$, but the estimates in ~\cite{Met-Tat} only apply when lower order terms are present if we have $s=2$ or $s=1$, see ~\cite[Corollary 5 and Theorem 6]{Met-Tat}. This leaves $d=4$ as the only option, as here $\frac{d}{2}=2$.  
 \end{rem}

\vspace{\baselineskip}

\section{Uniqueness}\label{uniqueness}
We use the extrinsic formulation ~\eqref{wm extrinsic local} of the wave maps system to prove uniqueness. The argument given for uniqueness in ~\cite{Shat-Stru WM} adapts perfectly to our case and we reproduce it below for completeness. 

Suppose that $(u,\dot{u})$ and $(v,\dot{v})$ are two solutions to ~\eqref{wm extrinsic local} of class $H^2 \times H^1 ((\R^4,g);TN)$ such that 

\begin{align}\label{same data}
(u, \dot{u})\vert_{t=0} =  (v, \dot{v})\vert_{t=0}
\end{align}

In addition, assume that 

\begin{align}\label{finite L^8}
\|du\|_{L^{2}_t L^8_x} < \infty , \quad \|dv\|_{L^{2}_t L^8_x} < \infty 
\end{align}

Set $w = u-v$. Then $w$ satisfies 

\begin{align*}
\Box_{\eta} w =  -\eta^{\al \be}[S(u)-S(v)](\p_{\al}u, \p_{\be} u) - \eta^{\al \be}S(v)(\p_{\al}u + \p_{\al} v, \p_{\be} w)
\end{align*}
By considering the pairing $\ang{ \Box_{\eta} w, \dot{w}}$ and integrating over $M$ we obtain

\begin{align*}
\frac{1}{2} \frac{d}{dt} \|dw\|_{L^{2}}^2 &= \int_{\R^4} \ang{ \eta^{\al \be}[S(u)-S(v)](\p_{\al}u, \p_{\be} u), \dot{w}} \sqrt{\abs{g}} \,dx \\
&\quad + \int_{\R^4} \ang{\eta^{\al \be}S(v)(\p_{\al}u + \p_{\al} v, \p_{\be} w), \dot{w}} \sqrt{\abs{g}} \, dx\\
\\
& = I(t) + II(t)
\end{align*} 
Using that $S$ and all of its derivatives are bounded we have 

\begin{align*}
\abs{I(t)} &\lesssim \int_{\R^4} \abs{du}^2 \abs{w} \abs{dw}\,dx\\
&\lesssim \|du\|_{L^8}^2\|w\|_{L^4} \|dw\|_{L^2} \\
&\lesssim \|du\|_{L^8}^2\|dw\|_{L^2}^2
\end{align*}
with the last inequality following from the Sobolev embedding $\dot{H}^1(\R^4) \hookrightarrow L^4(\R^4)$.

To estimate $II(t)$, we exploit the fact the $S(u)(\cdot, \cdot) \in (T_uN)^{\perp}$ which gives 
$\ang{ S(u)(\cdot, \cdot), u_t} = \ang{S(v)(\cdot, \cdot), v_t}=0$. This implies that we can rewrite

\begin{align*}
\abs{\ang{\eta^{\al \be}S(v)(\p_{\al}u + \p_{\al} v, \p_{\be} w), \dot{w}}} &= \abs{\ang{\eta^{\al \be}S(v)(\p_{\al}u + \p_{\al} v, \p_{\be} w), \dot{u}}} \\
\\
&=\abs{\ang{ \eta^{\al \be}[S(v)-S(u)](\p_{\al}u + \p_{\al} v, \p_{\be} w), \dot{u}}}\\
\\
&\le \abs{\ang{ \eta^{\al \be}[S(v)-S(u)](\p_{\al}u , \p_{\be} w), \dot{u}}} \\
&\quad+\abs{\ang{ \eta^{\al \be}[S(v)-S(u)]( \p_{\al} v, \p_{\be} w), \dot{u}}}\\
\\
&\lesssim (\abs{du}^2 +\abs{dv}^2) \abs{w} \abs{dw}
\end{align*} 
Hence we have 

\begin{align*}
\abs{II(t)} &\lesssim (\|du\|^2_{L^8} +\|dv\|_{L^8}^2)\|w\|_{L^4}\|dw\|_{L^2}\lesssim (\|du\|^2_{L^8} +\|dv\|_{L^8}^2)\|dw\|_{L^2}^2
\end{align*}
Putting this together we have 

\begin{align*}
\frac{1}{2} \frac{d}{dt} \|dw\|_{L^{2}}^2 &\lesssim (\|du\|_{L^8}^2 +\|dv\|_{L^8}^2)\|dw\|_{L^2}^2
\end{align*}
Integrating in $t$ and applying Gronwall's inequality gives us the uniform estimate

\begin{align*}
\|dw\|_{L^{\infty}_tL^{2}_x}^2 \le \|dw(0)\|_{L^2}^2 \cdot \textrm{exp} (C(\|du\|^2_{L^2_tL^8_x} +\|dv\|^2_{L^2_tL^8_x}))
\end{align*}
which implies uniqueness since $dw(0) =0$.

\vspace{\baselineskip}

\section{Coulomb Frame \& Elliptic Estimates}\label{Coulomb Frame}
We follow ~\cite{Shat-Stru WM} by exploiting the gauge invariance of the wave maps problem and rephrasing the wave maps equation in terms of the Coulomb frame. As discussed in Section ~\ref{geometric framework}, we can, without loss of generality, assume that $TN$ is parallelizable, and we choose a global orthonormal frame $\ti{e} = \{\ti{e_1},\dots,\ti{e_n}\}$. If $u: (\ti{M}, \eta) \to (N,h)$  is a smooth map, then we can pull back $\ti{e}$ to an orthonormal frame $\bar{e}= \ti{e}\circ u$ of $u^*TN$. Now, let $B: \R\times M\longrightarrow SO(n)$. With $B$ we can rotate this frame over each point $z\in \R\times M$ and obtain a new frame $e=(e_1, \dots, e_n)$,  with $e_a$ given by

\begin{align*}
e_a = B^b_a \bar{e}_b
\end{align*}
Observe that we can express the $u^*TN$-valued $1$-form $du$ in this new frame by finding $1$-forms $q^a=q^a_{\al} dx^{\al}$ where  $q^a_{\al} =u^*h(\p_{\al} u, e_a)$, and writing

\begin{align}\label{def of q}
du = q^a e_a
\end{align}
For this frame $e$ we have the associated connection form $A$. $A= (A^a_b)$ is a matrix of $1$-forms obtained in the following way. Given the frame $e$, we obtain for each $s\in\R$ a map 

\begin{align*} 
De_a: \Gamma(T(\{s\}\times & M))\longrightarrow \Gamma(u^*TN) \\
&X\longmapsto D_Xe_a
\end{align*}
 where $D$ is the pull back connection on $u^*TN$ and where for a vector bundle $E\rightarrow M$, $\Gamma(E)$ denotes the space of smooth sections. Equivalently, we can view $De_a$ as a section of $T^*M \otimes u^*TN$. We can express this map in terms of the connection form $A$ which can be viewed as the matrix of $1$-forms so that 
 
 \begin{gather*}
 De_a = A^b_a \otimes e_b\\
 De_a(X) = D_Xe_a= A^b_a(X) e_b
 \end{gather*}
Observe that this is the same as viewing  $De_a$  as a $\binom{1}{1}$-tensor on $T^*M \otimes u^*TN\to M$ in the sense that $De_a: TM \times u^*T^*N \to \R$ is a bilinear map over $C^{\infty}(M)$ . Then we have that 
 
 \begin{align*}
 A^b_a = u^*h(A^c_a \otimes e_c, e_b)
 \end{align*}
where $u^*h$ is the metric on $u^*TN$.   In local coordinates,  $A^b_a$ is given by $A^b_{a, \al}dx^{\al}$ where the  coefficients of $A^b_a$ are defined by  $A^b_{a, \al} = A^b_a(\p_{\al})$.  Hence if $X$ is given in local coordinates by $X= X^{\al}\p_{\al}$ we have that $D_Xe_a = X^{\al}A_{ a, \al}^{b}e_b$. 

One should also note that for a fixed coordinate $\al$, the matrix $(A^{a}_{b, \al})$ is antisymmetric. That is, $A^{a}_{b, \al} = -A^b_{a, \al}$.  To see this, simply differentiate the orthogonality condition of our orthonormal frame,  $h( e_a, e_b) = \de_{ab}$. This gives

\begin{align*}
0 &= D\left( h( e_a, e_b)\right) \\
&= h(De_a, e_b) + h(e_a, De_b)\\
&= A_a^b + A_b^a
\end{align*}

The curvature tensor, F,  on $u^*TN$ can be represented in term of the connection form $A$. Viewed as a  $2$-form,  $F$ is given by $F= dA+ A\wedge A$. We can also represent $F$ in terms of the curvature tensor on $TN$. In local coordinates,  $F$ is given by $F_{\al \be} = R(u)(\p_{\al }u, \p_{\be} u)$.

As in ~\cite[Lemma 4.1.3]{Hel},  we choose our rotation $B$ so that at for each $s\in\R$,  $B(s, \cdot)$  minimizes  the functional 

\begin{align*}
\Lambda(B(s)) &= \int_M \sum_{a,b=1}^n g^{-1}(A^b_a(s), A^b_a(s)) \, dvol_g\\
&= \int_M \sum_{a,b=1}^n g^{\al \be}A^b_{a,\al}(s) A^b_{a,\be}(s) \, \sqrt{\abs{g}}\, dx
\end{align*}
This gives us a frame $e$ that we call the Coulomb frame.  The Euler-Lagrange equations for this minimization problem are given by 
 
\begin{align}\label{EL eqns}
 \frac{1}{\sqrt{\abs{g}}} \p_{\al} (\sqrt{\abs{g}} g^{\al\be}A_{\be}) = 0
 \end{align}
 The above equation implies that $\de A= 0$ since the exterior co-differential, $\de$, on $1$-forms is given in local coordinates by 
 
 \begin{align}\label{E-L for A}
 -\de A =  \frac{1}{\sqrt{\abs{g}}} \p_{\al} (\sqrt{\abs{g}} g^{\al\be}A_{\be}) = 0
 \end{align}
 Since the Hodge Laplacian $\Delta$ on $M$ is given by $\Delta= d\de + \de d$, ~\eqref{E-L for  A} implies the following differential equation of $1$-forms for $A$ 
 
 \begin{align*} 
 \Delta A = \de d A
 \end{align*}
 Using the fact that the curvature form $F$ satisfies $F= dA + A \wedge A$ we can rewrite the above equation for $A$ as 

 \begin{align}\label{eqn for A}
 \Delta A = \de(F- A\wedge A)
 \end{align}
 In local coordinates we can write this in components as
 
 \begin{align}\label{eqn for A local}
 (\Delta A)_{\ga} = -[\n^{\al}(F-A\wedge A)]_{\al \ga}
 \end{align}
 where $\n^{\al} = g^{\al \be}\n_{\be} $ and $\n$ denotes the Levi-Civita connection on $M$.

Observe that ~\eqref{eqn for A local} can be written as system of elliptic equations for the components of $A$ in local coordinates on $M$. We record this fact in the following lemma:
  
  \begin{lem} The components of A satisfy the following system of elliptic equations 
 
  \begin{multline}\label{local elliptic eqn}
  g^{i j} \p_{i }\p_{j}A_{\ga} - g^{i j} \Gamma^{k}_{i j} \p_{\ga} A_{k} +\p_{\ga} g^{i j} \p_{j} A_{i} - \p_{\ga}( g^{ij} \Gamma^{k}_{ij} ) A_{k} \\= g^{i j} \p_{j} \left( F_{i \ga} - [A_{i}, A_{\ga}] \right)
  \end{multline}
  where the $\Gamma^k_{i j}= \frac{1}{2}g^{km}\left( \p_i g_{mj} + \p_j g_{im}- \p_m g_{ij}\right)$ denote the Christoffel symbols on ~$M$. 
  \end{lem}
  
  \begin{proof} We first expand the left-hand side of ~\eqref{eqn for A local} 
  
  \begin{align*}
  (\Delta A)_{\ga} &= (d\de A)_{\ga} + (\de d A)_{\ga}\\
  \\
  &= -\p_{\ga}(g^{i j} (\n_{j} A)_{i})  - g^{i j} (\n_{j} dA)_{i \ga}\\
  \\
  &= -\left(\p_{\ga} g^{ij}\right)(\n_{j} A)_{i} - g^{ij} \p_{\ga} \left(\p_{j} A_i -  \Gamma^{k}_{ij}A_k\right)\\
  &\quad  - g^{ij}\left(\p_j(dA)_{i\ga} - \Gamma^k_{ij}(dA)_{k \ga} - \Gamma^k_{\ga j}(dA)_{i k}\right)\\
  \\
  &= -g^{ij}\p_i \p_j A_{\ga} -\left(\p_{\ga}g^{ij}\right)(\n_{j} A)_{i}  + g^{ij}\p_{\ga}\left( \Gamma^{k}_{ij}A_k\right)\\
  &\quad + g^{ij}\Gamma^k_{ij}(dA)_{k \ga} + g^{ij}\Gamma^k_{\ga j}(dA)_{i k}
  \end{align*}
  Similarly, we expand the right-hand side of ~\eqref{eqn for A local} 
  
  \begin{align*}
  -[\n^{i}(F-A\wedge A)]_{i \ga}&= -g^{ij}\p_j\left(F_{i \ga} - \left[ A_i, A_{\ga}\right] \right) + g^{ij}\Gamma^k_{ij}\left(F_{k \ga} - \left[A_k, A_{\ga}\right] \right) \\
  &\quad + g^{ij} \Gamma^k_{j \ga}\left( F_{ik} - \left[ A_i, A_k\right]\right)
  \end{align*}
  Equating the left and right hand sides and recalling that $(dA)_{ij} = F_{ij} - [A_i, A_j]$ we have 
  
  \begin{align*}
  g^{ij}\p_i \p_j A_{\ga} +\left(\p_{\ga}g^{ij}\right)(\n_{j} A)_{i}  - g^{ij}\p_{\ga}\left( \Gamma^{k}_{ij}A_k\right) = g^{ij}\p_j\left(F_{i \ga} - \left[ A_i, A_{\ga}\right] \right)
  \end{align*}
  which is exactly ~\eqref{local elliptic eqn}.
  \end{proof}

\vspace{\baselineskip}

\subsection{Connection Form Estimates}\label{connection form estimates} With the metric $g$ as in ~\eqref{g}--\eqref{p^2 g} and $\ep$ small enough,  we can use the elliptic system  ~\eqref{local elliptic eqn} to establish a variety of estimates for the connection form $A$. In particular, we can prove the following proposition which will be essential when deriving a priori estimates for wave maps. 

\begin{prop}\label{main elliptic estimates} Let $(N,h)$ be a $n$-dimensional manifold smoothly embedded in $\R^m$ with bounded geometry and a bounded parallelizable structure.  Let $u: (\R\times \R^4, \eta) \to (N,h)$ be a smooth map with $\eta=\textrm{diag}(-1, g)$ and $g$ as in ~\eqref{g}--\eqref{p^2 g}. Moreover, assume the bootstrap hypothesis,  
\begin{align}\label{bootstrap}
\sup_{t\in[0,T)}\|du\|_{\dot{H}^1} \lesssim \ep_0
\end{align}
Then, for each $t\in \R$, there exists a unique frame $e=(e_1, \dots, e_n)$ for $u^*TN$ with the associated connection form, $A$,  satisfying the uniform-in-time estimates
\\

\begin{list}{(\roman{parts})}{\usecounter{parts}}
\item $\ds{\norm{A}_{L^4} \lesssim \|du\|_{H^1} \lesssim \ep_0 \label{A L4}}$
\\
\item  $\ds{\norm{A}_{\dot{W}^{1,\frac{8}{3}}} \lesssim \|du\|_{L^8} \|du\|_{\dot{H}^1}}$
\\
\item $\ds{\norm{A}_{\dot{W}^{2,\frac{8}{5}}} \lesssim \|du\|_{L^8} \|du\|_{\dot{H}^1}}$
\\
\item $ \|A\|_{L^{\infty}} \lesssim \|du\|_{L^{8,2}(\R^4)}^2$
\end{list}
\vspace{\baselineskip}
 as long as $\ep_0$ is small enough. Also, the frame $e$, and hence $A$, depend continuously on $t$. Above, $L^{8,2}=L^{8,2}(\R^4)$ denotes the Lorentz space.
\end{prop}

The estimates are deduced via a perturbative method as the assumptions in~\eqref{g}--\eqref{p^2 g} imply that the left hand side of ~\eqref{local elliptic eqn} is a slight perturbation of the flat Laplacian on $\R^4$. To simplify notation, in what follows we consider an elliptic operator of the form 
 
 \begin{align}\label{def L}
 L:= g^{ij}\p_i \p_j + b^j \p_j + c
 \end{align}
and the elliptic system

\begin{align}\label{L}
L A_{\ell}= g^{i j} \p_{j} G_{i \ell} 
\end{align}
where $G_{i\ell} :=  F_{i \ell} - [A_{i}, A_{\ell}] $, and $b$ and $c$ satisfy

\begin{align}
\|b\|_{L^{4,1}(\R^4)} \lesssim \ep \label{b}\\ \notag
\\
\|\p b\|_{L^{2,1}(\R^4)} \lesssim \ep \label{p b}\\ \notag
\\
\|c\|_{L^{2,1}(\R^4)} \lesssim \ep \label{c}
\end{align}
Since $\Gamma_{ij}^k= \frac{1}{2}g^{k\ell}(\p_i g_{\ell j} + \p_j g_{i\ell} - \p_{\ell} g_{ij})$, it is clear that the left-hand side of ~\eqref{local elliptic eqn} is essentially of  this form.

 We begin by recalling some basic elliptic estimates. Let $g_0$ denote the Euclidean metric on $\R^4$ and let $L_0 := g_0^{ij} \p_i \p_j$ denote the flat Laplacian on $\R^4$. Then we have 
 
 \begin{align}\label{basic elliptic}
 \|A\|_{\dot{W}^{s+2,p}} \lesssim \|L_0 A\|_{\dot{W}^{s,p}}
 \end{align}
 for every $s \in \R$ and for every $1<p<\infty$. 
With \eqref{basic elliptic} we can prove the following elliptic estimates for the connection form $A$. 
 
 \begin{lem}  \label{elliptic estimates}
 Let $A$ be the connection form associated to the Coulomb frame $e$. Then, if $\ep$ is small enough,  we have the following uniform-in-time estimates
\begin{list}{(\roman{parts})}{\usecounter{parts}}
\item $\ds{ \|A\|_{\dot{W}^{1,p}} \lesssim \|[A, A]\|_{L^p} + \|F\|_{L^p} }$ if $1<p<4$
\\
\item $\ds{ \|A\|_{\dot{W}^{2,p}} \lesssim \|[A, A]\|_{\dot{W}^{1,p}}  + \|F\|_{\dot{W}^{1,p}} }$ if $1<p<2$. 
\end{list}
where $F$ denotes the curvature tensor on $u^*TN$.  
\end{lem}

\begin{proof} Let $L_0$ and $L$ be defined as above and write $LA = L_0A + (L-L_0)A$. Hence, 

\begin{align*}
\|LA\|_{\dot{W}^{s,p}} \ge \|L_0A\|_{\dot{W}^{s,p}}  - \|(L-L_0)A\|_{\dot{W}^{s,p}} 
\end{align*}
We can use ~\eqref{basic elliptic} to obtain

\begin{align}\label{elliptic}
\|A\|_{\dot{W}^{s+2,p}} \lesssim  \|LA\|_{\dot{W}^{s,p}} + \|(L-L_0)A\|_{\dot{W}^{s,p}}
\end{align}
for every $s\in \R$  and for $1<p<\infty$. To prove ~$(i)$, set $s=-1$ above to get

\begin{align*}
\norm{A}_{\dot{W}^{1,p}} &\lesssim  \|LA\|_{\dot{W}^{-1,p}} + \|(L-L_0)A\|_{\dot{W}^{-1,p}}\\ \notag
\\  \notag
&\lesssim \norm{g^{-1}\, \p G}_{\dot{W}^{-1,p}} + \norm{b\,  \p A}_{\dot{W}^{-1,p}} \\ 
\\ \notag
&\quad+ \norm{c A}_{\dot{W}^{-1,p}} + \left\|(g^{-1}-g_0^{-1})\p^2A\right\|_{\dot{W}^{-1,p}}\notag
\end{align*}
We claim that 

\begin{align*}
\norm{g^{-1}\, \p G}_{\dot{W}^{-1,p}} \lesssim \norm{G}_{L^p}
\end{align*}
This follows from the dual estimate

\begin{align}\label{dual}
\norm{g^{-1} f}_{\dot{W}^{1,p'}}\lesssim \norm{f}_{\dot{W}^{1,p'}}
\end{align}
To prove ~\eqref{dual} observe that we have 

\begin{align*}
\norm{g^{-1} f}_{\dot{W}^{1,p'}}&\lesssim \norm{\p(g^{-1} f)}_{L^{p'}}\\
\\
&\lesssim \norm{(\p g^{-1}) f}_{L^{p'}} + \norm{ g^{-1} (\p f)}_{L^{p'}}\\
\\
&\lesssim \norm{ \p g^{-1}}_{L^4} \norm{f}_{L^r} + \norm{g^{-1}}_{L^{\infty}} \norm{\p f}_{L^{p'}}\\
\\
&\lesssim \norm{f}_{\dot{W}^{1,p'}}
\end{align*}
where the last inequality follows from ~\eqref{p g} and the Sobolev embedding $\dot{W}^{1,p'} \hookrightarrow L^r$ since we have $\frac{1}{r} = \frac{1}{p'} - \frac{1}{4}$. 
Next, we assert that 

\begin{align*}
\left\|(g^{-1}-g_0^{-1})\p^2A\right\|_{\dot{W}^{-1,p}} & \lesssim 
\ep \norm{\p^2 A}_{\dot{W}^{-1,p}}\lesssim \ep\norm{A}_{\dot{W}^{1,p}}
\end{align*}
Again, this follows from a duality argument. Observe that 

\begin{align*}
\|(g^{-1} -g_0^{-1})f \|_{\dot{W}^{1,p^{\prime}}} &\lesssim \| \p(g^{-1} -g_0^{-1}) f\|_{\dot{W}^{1, p^{\prime}}} + \| (g^{-1} -g_0^{-1}) \p f\|_{\dot{W}^{1, p^{\prime}}}\\
\\
&\lesssim \|\p g^{-1}\|_{L^4} \|f\|_{L^r} + \|(g^{-1} -g_0^{-1})\|_{L^{\infty}}\|\p f\|_{L^{p^{\prime}}}\\
\\
&\lesssim \ep\|f\|_{\dot{W}^{1, p^{\prime}}}
\end{align*}
where the last inequality is again due to \eqref{g}, \eqref{p g}, and Sobolev embedding since $\frac{1}{r} = \frac{1}{p'} - \frac{1}{4}$. 
To estimate $\|b\, \p A\|_{\dot{W}^{-1,p}}$, we use Sobolev embedding, H\"older's inequality and ~\eqref{b}. Indeed,

\begin{align*}
\norm{b\,  \p A}_{\dot{W}^{-1,p}} & \lesssim \|b\, \p A\|_{L^s}\\
\\
&\lesssim \|b\|_{L^4} \|\p A\|_{L^{p}} \\
\\
& \lesssim \ep \norm{A}_{\dot{W}^{1,p}}
\end{align*} 
where $\frac{1}{p}= \frac{1}{s}-\frac{1}{4}$. Finally, we show that

\begin{align*}
\norm{c A}_{\dot{W}^{-1,p}}\lesssim \ep \norm{A}_{\dot{W}^{1,p}}
\end{align*}
To see this, we again use Sobolev embedding and  ~\eqref{c} to obtain
\begin{align*}
\norm{c\,A}_{\dot{W}^{-1,p}} &\lesssim \|c\,A\|_{L^s}\\
\\
& \lesssim \|c\|_{L^2} \|A\|_{L^r}\\
\\
& \lesssim \ep \|A\|_{\dot{W}^{1,p}}
\end{align*}
with $\frac{1}{p} = \frac{1}{s} - \frac{1}{4}$, $\frac{1}{s} = \frac{1}{2} + \frac{1}{r}$, and $ \frac{1}{r}= \frac{1}{p} - \frac{1}{4}$.  
Putting this all together we  are able to conclude that 

\begin{align*}
\norm{A}_{\dot{W}^{1,p}} \lesssim \|G\|_{L^p} + \ep \norm{A}_{\dot{W}^{1,p}}
\end{align*}
For $\ep$ small enough, this implies ~$(i)$, since $G= F-A\wedge A$. 

\vspace{\baselineskip}

To prove ~$(ii)$ we set $s=0$ in ~\eqref{elliptic}, and use ~\eqref{g},  ~\eqref{b}, ~\eqref{c}, and Sobolev embedding to obtain

\begin{align*}
\norm{A}_{\dot{W}^{2,p}} &\lesssim \norm{g^{-1}\, \p G}_{L^p} + \norm{b\,  \p A}_{L^p}+ \norm{c A}_{L^p} + \left\|(g^{-1}-g_0^{-1})\p^2A\right\|_{L^p}\\
\\
&\lesssim \norm{g^{-1}}_{L^{\infty}} \|\p G\|_{L^p} + \|b\|_{L^4}\|\p A\|_{L^s} + \|c\|_{L^2}\|A\|_{L^r}\\
&\quad+ \norm{g^{-1} -g_0^{-1}}_{L^{\infty}} \|\p^2A\|_{L^p}\\
\\
&\lesssim \|G\|_{\dot{W}^{1,p}} + \ep \|A\|_{\dot{W}^{2,p}}
\end{align*}
where $\ds{\frac{1}{s}= \frac{1}{p} -\frac{1}{4}}$ and $\ds{\frac{1}{r} = \frac{1}{p} - \frac{2}{4}}$. This proves ~$(ii)$ as long as $\ep$ is small enough. 
\end{proof}

\vspace{\baselineskip}

With the elliptic estimates in Lemma ~\ref{elliptic estimates} we can prove Proposition ~\ref{main elliptic estimates} ~$(i)$, ~$(ii)$ and ~~$(iii)$.

\begin{proof}[Proof of Proposition ~\ref{main elliptic estimates} ~$(i)$] This will follow from Lemma ~\ref{elliptic estimates} ~$(i)$ with $p=2$, a contraction argument at one fixed time, and then a bootstrap argument to conclude the uniform-in-time estimates. We note that this argument also proves the existence of a unique Coulomb frame $e$ with the associated connection form $A$ having small $L^4$ norm. 

To carry out the contraction argument we fix a time $t_0$ and we set $X$ to  be the space

\begin{align*}
X:= \{A\in \dot{H}^1 \cap L^4\}
\end{align*}
with the norm 
\begin{align*}
\|A\|_{X} := \|A\|_{L^4} + \|A\|_{\dot{H}^1}
\end{align*}
Of course by Sobolev embedding we have $\|A\|_{X} \lesssim \|A\|_{\dot{H}^1}$. We set $X_{\ep_0}$ to be 

\begin{align*}
X_{\ep_0} := \{A\in X : \|A\|_{X} \le \ep_0\}
\end{align*}
Define a map $\Phi$ that associates to each $\ti{A} \in X_{\ep_0}$ the solution $A$ to the linear elliptic problem
\begin{align}\label{linear elliptic}
LA_{\ell} = g^{ij}\p_i(F_{j \ell} - [\ti{A}_j, \ti{A}_{\ell}])
\end{align}
The existence of such a solution follows easily by the method of continuity, the key estimate here being
\begin{align*}
\|A\|_{\dot{H}^1} \lesssim \|LA\|_{\dot{H}^{-1}}
\end{align*}
which was obtained in the course of proving Lemma ~\ref{elliptic estimates} ~$(i)$ with  $p=2$.  We will show that if $\ep_0$ and $\|du\|_{\dot{H}^1}$ are small enough, then $\Phi: X_{\ep_0} \to X_{\ep_0}$ and that $\Phi$ is a contraction mapping on this space. To see that $\Phi: X_{\ep_0} \to X_{\ep_0}$ we use  
Sobolev embedding and Lemma ~\ref{elliptic estimates} ~$(i)$ to obtain 

\begin{align*}
\|A\|_{X} \lesssim \|A\|_{\dot{H}^{1}} \lesssim \|[\ti{A},\ti{A}]\|_{L^2}+ \|F\|_{L^2}
\end{align*}
Recall that we can write $F_{\al \be} = R(u)(\p_{\al}u, \p_{\be}u)$ where $R$ is the Riemannian curvature tensor on $N$. Hence 

\begin{align*}
\|A\|_{X} &\lesssim  \|\ti{A}\|^2_{L^4}+ \|R\|_{L^{\infty}}\|du\|_{L^4}^2\\
\\
&\le  C_1\|\ti{A}\|^2_{X}+ C_2\|du\|_{\dot{H}^1}^2 \le \ep_0
\end{align*}
as long as $\ep_0$ and $\|du\|_{\dot{H}^1}$ are small enough. Next we show that $\Phi: X_{\ep_0} \to X_{\ep_0}$ is a contraction mapping. Let $\ti{A}^1, \ti{A}^2 \in X_{\ep_0}$ and let $A^1, A^2$ be the associated solutions to ~\eqref{linear elliptic}. Then $A^1-A^2$ is a solution to 

\begin{align*}
L(A_{\ell}^1-A_{\ell}^2) = g^{ij}\p_i([\ti{A}^1_j, \ti{A}^1_{\ell}] -[\ti{A}^2_j, \ti{A}^2_{\ell}])
\end{align*}
and hence we have estimates 

\begin{align*}
\|A^1-A^2\|_{X} \lesssim \|A^1- A^2\|_{\dot{H}^1}& \lesssim \|[\ti{A}^1, \ti{A}^1] -[\ti{A}^2, \ti{A}^2]\|_{L^2}\\
\\
&\lesssim \|\ti{A}^1-\ti{A}^2\|_{L^4}\|\ti{A}^1\|_{L^4} + \|\ti{A}^1-\ti{A}^2\|_{L^4}\|\ti{A}^2\|_{L^4}\\
\\
&\lesssim \ep_0 \|\ti{A}^1 -\ti{A}^2\|_{X}
\end{align*}
which proves that $\Phi$ is a contraction. Hence $\Phi$ has a unique fixed point $A=A({t_0})$ which solves ~\eqref{L} such that 
\begin{align*}
\|A({t_0})\|_{L^4} \lesssim \ep_0
\end{align*}
To obtain this estimate for all times $t$ with a uniform constant we again use Lemma ~\ref{elliptic estimates} ~~$(i)$ with $p=2$ to obtain for any time

\begin{align*}
\|A\|_{L^4} \lesssim \|A\|_{\dot{H}^1} &\lesssim \|[A,A]\|_{L^2} + \|F\|_{L^2}\\
\\
&\lesssim \|A\|_{L^4}^2 + \|du\|_{\dot{H}^1}^2
\end{align*}
As long as $\|du\|_{\dot{H}^1}$ is small enough we can use a bootstrap argument, with $\|A(t_0)\|\lesssim \ep_0$ as our base case, to absorb the $\|A\|_{L^4}^2$ term on the left hand side and obtain 

\begin{align*}
\|A\|_{L^4} \lesssim \ep_0
\end{align*}
for all times $t$, as desired. 
\end{proof}

\vspace{\baselineskip}

  \begin{proof}[Proof of Propostion ~\ref{main elliptic estimates} ~$(ii)$ and ~$(iii)$.]  To prove ~$(ii)$ we set  $p=\frac{8}{3}$ in Lemma ~\ref{elliptic estimates} ~$(i)$, giving 
  
 \begin{align*}
\|A\|_{\dot{W}^{1,\frac{8}{3}}} \lesssim \|[A, A]\|_{L^{\frac{8}{3}}}  + \|F\|_{L^{\frac{8}{3}}}
\end{align*} 
  First we claim that  $\ds{\|[A, A]\|_{L^{\frac{8}{3}}}  \lesssim \ep \|A\|_{\dot{W}^{1,\frac{8}{3}}}}$ and this term can thus be absorbed on the left-hand side above. Indeed,

\begin{align*}
\|[A, A]\|_{L^{\frac{8}{3}}} &\lesssim \|A\|_{L^4} \|\p A\|_{L^{8}}\\
\\
&\lesssim \ep_0 \|A\|_{\dot{W}^{1,\frac{8}{3}}}
\end{align*}
where the last inequality follows from Sobolev embedding and the previous estimate $\ds{\|A\|_{L^4} \lesssim \ep}$. Next we recall that $F= R(u)(du, du)$ and hence we have 

\begin{align*}
\|F\|_{L^{\frac{8}{3}}} \lesssim \|du\|_{L^8}\|du\|_{L^4} \lesssim \|du\|_{L^8}\|du\|_{\dot{H}^1}
\end{align*}
Putting this together implies gives 

\begin{align*}
\|A\|_{\dot{W}^{1,\frac{8}{3}}} \lesssim \|du\|_{L^8} \|du\|_{\dot{H}^1}
\end{align*}
as long as $\ep$ is small enough.
 
 To prove ~~$(iii)$ we proceed in a similar fashion. We set $p=\frac{8}{5}$ in Lemma ~\eqref{elliptic estimates} ~$(ii)$. This gives

\begin{align*}
\|A\|_{\dot{W}^{2,\frac{8}{5}}} \lesssim \|[A, A]\|_{\dot{W}^{1,\frac{8}{5}}}  + \|F\|_{\dot{W}^{1,\frac{8}{5}}} 
\end{align*}
First we observe that $\ds{\|[A, A]\|_{\dot{W}^{1,\frac{8}{5}}} \lesssim \ep \|A\|_{\dot{W}^{2,\frac{8}{5}}}}$ and this term can thus be absorbed on the left-hand side above. In fact, 

\begin{align*}
\|[A, A]\|_{\dot{W}^{1,\frac{8}{5}}} &\lesssim \|A \p A\|_{L^{\frac{8}{5}}}\\
\\
&\lesssim \|A\|_{L^4} \|\p A\|_{L^{\frac{8}{3}}}\\
\\
&\lesssim \ep_0 \|A\|_{\dot{W}^{2,\frac{8}{5}}}
\end{align*}
where the last inequality follows from Sobolev embedding and the previous estimate $\ds{\|A\|_{L^4} \lesssim \ep}$. Next observe that 

\begin{align*}
\p_{\ga} F_{\al \be} = (\p R(u))(\p_{\ga} u, \p_{\al}u, \p_{\be} u) + R(u)( \p_{\ga}\p_{\al}u, \p_{\be} u) + R(u)( \p_{\al}u, \p_{\ga}\p_{\be} u)
\end{align*}
Hence by Sobolev embedding and the assumption that $\|du\|_{\dot{H}^1} \lesssim \ep_0$,

\begin{align*} 
\|\p F\|_{L^{\frac{8}{5}}} &\lesssim \|\p R(u)\|_{L^{\infty}} \|du\|_{L^4} \|du\|_{L^4} \|du\|_{L^8} + \|R\|_{L^{\infty}}\|\p du\|_{L^2} \|du\|_{L^{8}} \\
\\
&\lesssim \|du\|_{\dot{H}^1} \|du\|_{L^8}
\end{align*}
Putting this all together we have for small enough $\ep_0$ that

\begin{align*}
\|A\|_{\dot{W}^{2,\frac{8}{5}}} \lesssim  \|du\|_{\dot{H}^1} \|du\|_{L^8} \lesssim \|du\|_{L^8}
\end{align*}
establishing ~~$(iii)$. 
\end{proof}

\vspace{\baselineskip}

To prove the pointwise estimates for the connection form in Propostion ~\ref{main elliptic estimates} ~$(iv)$, we will need a few facts about Lorentz Spaces, $L^{p,r}(\R^4)$, including Sobolev embedding for Lorentz spaces and  the Calderon-Zygmund theorem for  Lorentz spaces. These facts, along with a few others, are  reviewed  in the appendix,  see Section ~\ref{Lorentz Spaces}.

Now, again let $L_0= g_0^{ij} \p_i \p_j$ be the flat Laplacian on $\R^4$  and let $K=L_0^{-1}$ be convolution with $k(x) = \ds{\frac{c}{\abs{x}^2}}$, the fundamental solution for $L_0$ in $\R^4$.  We can then write 

\begin{align*}
A= KLA + K(L_0-L)A
\end{align*}

In order to prove Proposition ~\ref{main elliptic estimates} ~$(iv)$, we will need the following preliminary estimates for  $\p A$. 

\begin{lem}\label{grad A} Let $A$ denote the connection form associated to the Coulomb frame as in Proposition \ref{main elliptic estimates}. Then, the following estimates hold uniformly in time: 
\begin{align*}
\|\p A\|_{L^{4,1}}  \lesssim \|du\|^2_{L^{8,2}} + \ep\|A\|_{L^{\infty}}
\end{align*}
\end{lem}

\begin{proof}
With $K$, $L$ and $L_0$ as above,  write

\begin{align*}
A = KLA + K(L_0-L) A
\end{align*}
Then

\begin{align*}
 KLA&= k\ast g^{ij}\p_i G_j \\
 &= (\p_ik)\ast g^{ij} G_j - k\ast (\p_ig^{ij}) G_j
 \end{align*}
Then, formally, we have 

\begin{align}
\p_{\al}( KLA) = (\p_{\al}\p_ik)\ast g^{ij} G_j -  (\p_{\al}k)\ast (\p_ig^{ij}) G_j\label{dkla}
\end{align}
Since, $\p_{\al} \p_i k $ is a Calderon-Zygmund kernel, we can use the Calderon-Zygmund theorem for Lorentz spaces, see Theorem ~\ref{CZ} below, and H\"{o}lder's inequality for Lorentz spaces, see Lemma ~\ref{Lor 2} ~$(i)$ below,  to obtain

\begin{align*}
\| (\p_{\al}\p_ik)\ast g^{ij} G_j\|_{L^{4,1}} &\lesssim \|g^{ij} G_j\|_{L^{4,1}} \\
\\
& \lesssim \|[A,A]\|_{L^{4,1}} + \|F\|_{L^{4,1}}\\
\\
&\lesssim \|A\|^2_{L^{8,2}} + \|du\|_{L^{8,2}}^2
\end{align*}
Using the fact that $L^{p,r} \subset L^{p,s}$ for $r<s$,  Sobolev embedding for Lorentz spaces, see Lemma ~\ref{Sob embedding Lorentz} below, and Proposition ~\ref{main elliptic estimates} ~~$(iii)$, we have 
 
 \begin{align*}
 \|A\|_{L^{8,2}} \lesssim \|A\|_{L^{8, \frac{8}{5}}} \lesssim \|A\|_{\dot{W}^{2,\frac{8}{5}}} \lesssim \|du\|_{L^{8}}\|du\|_{\dot{H}^1}
 \end{align*}
 Then using the bootstrap assumption that $\ds{\|du\|_{\dot{H}^1} \ll1}$ we can conclude that 
 
 \begin{align}\label{A L^{8,2}}
 \|A\|_{L^{8,2}} \lesssim \|du\|_{L^8} \lesssim \|du\|_{L^{8,2}}
 \end{align}
 Inserting this estimate above we can conclude that 
 
 \begin{align*}
 \| (\p_{\al}\p_ik)\ast g^{ij} G_j\|_{L^{4,1}}\lesssim \|du\|_{L^{8,2}}^2
 \end{align*}
Next, we can use Young's inequality for Lorentz spaces, see Lemma ~\ref{Lor 2} ~$(ii)$ below, to show

\begin{align*}
\|(\p_{\al}k)\ast (\p_ig^{ij}) G_j\|_{L^{4,1}} &\lesssim  \|\p_{\al}k\|_{L^{\frac{4}{3}, \infty}} \|\p_i g^{ij} G_j\|_{L^{2,1}}\\
\\
&\lesssim \|\p_i g^{ij} \|_{L^{4, \infty}} \|G_j\|_{L^{4,1}} \\
\\
&\lesssim \|A^2\|_{L^{4,1}} + \|F\|_{L^{4,1}}\\
\\
&\lesssim \|du\|_{L^{8,2}}^2
\end{align*}
where above we have used ~\eqref{p g}. 
Now, to deal with the error term ~$K(L_0-L)A$, write $(L_0-L) A = \ep^{ij} \p_i \p_j A- b^{j}\p_jA -  cA$ where $\ep^{ij}(x) = g_0^{ij}(x) - g^{ij}(x)$.  Then we have 

\begin{align*}
K(L_0-L) A&= k \ast \ep^{ij} \p_i \p_j A -  k\ast b^{j}\p_jA -k\ast cA\\
&= (\p_i k) \ast \ep^{ij} \p_j A - k \ast (\p_{i} \ep^{ij}) \p_j A-  k\ast b^{j}\p_jA - k\ast cA
\end{align*}
Hence, formally we have 

\begin{align}\label{derror}
\p_{\al}(K(L_0-L) A)&=  (\p_{\al}\p_i k) \ast \ep^{ij} \p_j A - (\p_{\al}k) \ast (\p_{i} \ep^{ij}) \p_j A \\ \notag
&\quad-( \p_{\al}k)\ast b^{j}\p_jA - (\p_{\al}k)\ast cA
\end{align}

And as before, we use the Calderon-Zygmund theorem on the first term on the right-hand side above to get

\begin{align*}
\|(\p_{\al}\p_i k) \ast \ep^{ij} \p_j A\|_{L^{4,1}} \lesssim \sum_i \|  \ep^{ij} \p_j A\|_{L^{4,1}} \lesssim \ep \|\p A\|_{L^{4,1}}
\end{align*}
We estimate the other three terms on the right-hand side of ~\eqref{derror} using Young's inequality for Lorentz spaces as follows 

\begin{align*}
\|(\p_{\al}k) \ast (\p_{i} \ep^{ij}) \p_j A\|_{L^{4,1}} &\lesssim \|\p_{\al} k \|_{L^{\frac{4}{3} , \infty}} \|(\p_{i} \ep^{ij}) \p_j A\|_{L^{2,1}}\\
\\
&\lesssim \|\p_i \ep^{ij}\|_{L^{4,\infty}}\|\p_j A \|_{L^{4,1}}\\
\\
&\lesssim \ep\|\p A\|_{L^{4,1}}
\end{align*}
the last inequality following from the fact that $\p_i \ep^{ij}= \p_i g^{ij} \in L^{4,\infty}$. We also have

\begin{align*}
\|\p_{\al}k \ast cA\|_{L^{4,1}}& \lesssim \|\p_{\al} k\|_{L^{\frac{4}{3}, \infty}} \|cA\|_{L^{2,1}}\\  
\\
&\lesssim \|c\|_{L^{2,1}}\|A\|_{L^{\infty}}\\
\\
&\lesssim \ep \|A\|_{L^{\infty}}
\end{align*}
Above we have used the fact that $\abs{\p_{\al}k} \sim \frac{1}{\abs{x}^3} \in L^{\frac{4}{3}, \infty}(\R^4)$ and $\|c\|_{L^{2,1}} \lesssim \ep$, see Lemma ~\ref{Lor 1}, ~~~$(iii)$ and ~\eqref{c}.  And lastly,

\begin{align*}
\|( \p_{\al}k)\ast b^{j}\p_jA\|_{L^{4,1}} &\lesssim \|\p_{\al} k \|_{L^{\frac{4}{3},\infty}}\|b^j \p_j A\|_{L^{2,1}}\\
\\
& \lesssim \|b\|_{L^{4}{\infty}}\|\p A\|_{L^{4,1}} \\
\\
&\lesssim \ep \|\p A\|_{L^{4,1}}
\end{align*}
which follows by \eqref{b}. Putting this all together establishes that 

\begin{align*}
\|\p A\|_{L^{4,1}} \lesssim \|du\|_{L^{8,2}}^2 + \ep \|A\|_{L^{\infty}} + \ep\| \p A \|_{L^{4,1}}
\end{align*}
which, for small $\ep$, implies that 

\begin{align*}
\|\p A\|_{L^{4,1}} \lesssim \|du\|_{L^{8,2}}^2 + \ep \|A\|_{L^{\infty}} 
\end{align*}
as desired. 
\end{proof}

\vspace{\baselineskip}

Now we are able to prove Proposition ~\eqref{main elliptic estimates} ~$(iv)$.

\begin{proof}[Proof of Proposition ~\ref{main elliptic estimates} ~$(iv)$]
Since $A= KLA + K(L_0-L)A$, it suffices to show that for every $t$ the following two estimates hold:

\begin{align}
\|KLA\|_{L^{\infty}} &\lesssim \|du\|_{L^{8,2}}^2 \label{KLA}\\
\notag\\
\|K(L_0-L)A\|_{L^{\infty}} &\lesssim \|du\|_{L^{8,2}}^2 + \ep \|A\|_{L^{\infty}}\label{KerrorA}
\end{align}
Observe that we can write

\begin{align*}
 KLA&= k\ast g^{ij}\p_i G_j \\
 &= (\p_i k)\ast g^{ij} G_j - k\ast (\p_i g^{ij}) G_j
 \end{align*}
By Lemma ~\ref{Lor 2} ~~$(iii)$, we have that 

\begin{align*}
\|(\p_i k)\ast g^{ij} G_j\|_{L^{\infty}} &\lesssim \|\p_i k\|_{L^{\frac{4}{3}, \infty}} \|g^{ij} G_j\|_{L^{4,1}}\\
\\
&\lesssim \|[A,A]\|_{L^{4,1}} + \|F\|_{L^{4,1}} \\
\\
&\lesssim \|A\|_{L^{8,2}}^2 + \|F\|_{L^{4,1}}^2\\
\\
&\lesssim \|du\|_{L^{8,2}}^2
\end{align*}
where we have used \eqref{A L^{8,2}} in the last inequality. Similarly,

\begin{align*}
\|k \ast (\p_i g^{ij}) G_j\|_{L^{\infty}} &\lesssim \|k\|_{L^{2,\infty}} \|\p_i g^{ij} G_j\|_{L^{2,1}}\\
\\
&\lesssim \|\p_i g^{ij}\|_{L^{4, \infty}} \|G_j\|_{L^{4,1}}\\
\\
&\lesssim \|G\|_{L^{4,1}}\\
\\
&\lesssim \|du\|_{L^{8,2}}
\end{align*}
This proves ~\eqref{KLA}. To establish the error estimate ~\eqref{KerrorA} we again write 

\begin{align} \notag
K(L_0-L) A &= k\ast \ep^{ij} \p_i \p_j A -  k\ast b^{j}\p_jA - k\ast cA \\ \notag
\\ \notag
&= \p_i k\ast \ep^{ij} \p_j A - k\ast(\p_i \ep^{ij}) \p_j A- (\p_j k)\ast b^j A + k\ast (\p_j b^{j})A\\\notag
&\quad- k\ast cA \\\notag
\\ \label{KL-L_0A}
&=  \p_i k\ast \ep^{ij} \p_j A - \p_j k \ast (\p_i \ep^{ij})  A + k \ast (\p_j \p_i \ep^{ij}) A- (\p_j k)\ast b^j A\\\notag
&\quad + k\ast (\p_j b^{j})A - k\ast cA 
\end{align}
 where as before $\ep^{ij} = g_0^{ij} - g^{ij}$. Now, we can use Lemma ~\ref{grad A} to control the first term on the right above
 
 \begin{align*}
 \|\p_i  k \ast \ep^{ij} \p_j A \|_{L^{\infty}} &\lesssim \|\p_i k \|_{L^{\frac{4}{3}}, \infty} \|\ep^{ij} \p_j A\|_{L^{4,1}}\\
 \\
 &\lesssim \ep \|\p A\|_{L^{4,1}}\\
 \\
 &\lesssim \|du\|_{L^{8,2}}^2 + \ep \|A\|_{L^{\infty}}
 \end{align*}
 The other terms in \eqref{KL-L_0A} are estimated as follows:
 
 \begin{align*}
 \|\p_j k \ast (\p_i \ep^{ij})A\|_{L^{\infty}}& \lesssim \|\p_j k\|_{L^{\frac{4}{3}, \infty}} \|(\p_i \ep^{ij}) A\|_{L^{4,1}} \\
 \\
& \lesssim \sum_{j} \|\p_i g^{ij}\|_{L^{4,1}} \|A\|_{L^{\infty}} \\
\\
&\lesssim \ep \|A\|_{L^{\infty}}
 \end{align*}
We also have 
\begin{align*}
 \|k \ast (\p_i \p_j\ep^{ij})\|_{L^{\infty}}& \lesssim \| k\|_{L^{2, \infty}} \|(\p_j \p_i \ep^{ij}) A\|_{L^{2,1}} \\
 \\
& \lesssim \|(\p_j\p_i \ep^{ij})\|_{L^{2,1}} \|A\|_{L^{\infty}} \\
\\
&\lesssim \ep \|A\|_{L^{\infty}}
 \end{align*}
The remaining terms are handled exactly in the same manner as these last two, using \eqref{b}, \eqref{p b} and \eqref{c} as needed.  This proves  ~\eqref{KerrorA}. Finally, putting everything together, we have 

\begin{align*}
\|A\|_{L^{\infty}} \lesssim \|KLA\|_{L^{\infty}} + \|K(L_0-L) A\|_{L^{\infty}} \lesssim \|du\|_{L^{8,2}}^2 + \ep\|A\|_{L^{\infty}}
\end{align*}
which, for $\ep$ small enough, gives 

\begin{align*}
\|A\|_{L^{\infty}} \lesssim \|du\|_{L^{8,2}}^2
\end{align*}
as claimed.
\end{proof}

\vspace{\baselineskip}

\subsection{Equivalence of Norms}\label{Equivalence of Norms} In this section we again set $(M, g) = (\R^4, g)$ with $g$ as in \eqref{g}-\eqref{p^k g}. Now that we have settled Proposition ~\ref{main elliptic estimates}, we can show that in the case that $e$ is the Coulomb frame, the extrinsic  $\dot{H}^s_e$ norms of $du$ are equivalent to the intrinsic $\dot{H}^s_i$ norms of $q= q^a e_a$ where $q$ is defined, as in ~\eqref{def of q}, by 
\begin{align*}
du= q^{a}e_a
\end{align*}

In the appendix, Section \ref{sobolev spaces}, we show using ~\eqref{g}-\eqref{p^k g}, that ~$\dot{H}^s_e((M,g);N)$ is equivalent to $\dot{H}_e^s((\R^4,g_0);N)$ and  that $\dot{H}^s_i((M,g);N)$ is equivalent to $\dot{H}_i^s((\R^4,g_0); N)$. Therefore, it suffices to ignore the perturbed metric $g$ on ~$\R^4$ and show that 

\begin{align}\label{H du =H q}
\|du\|_{\dot{H}_e^s((\R^4,g_0); N)} \simeq \|q\|_{\dot{H}_c^s((\R^4,g_0);N)}
\end{align}
This will follow from Proposition ~\ref{main elliptic estimates}.  We proceed exactly as in ~\cite[Section 4.3]{Shat-Stru WM}. We reproduce their argument here. 
For each $t$, since $e$ is an orthonormal frame,  we have 

\begin{align*}
\abs{du}^2 = \abs{q}^2 =\sum_{\al =0}^{4} \abs{q_{\al}}^2
\end{align*}
This implies that for $1\le p\le \infty$ that the $L^p$ norm of $du$ is well defined and independent of the choice of frame and coincides with the ``extrinsic'' $L^p$ norm of $du$.   However, this ``gauge'' independence is in general lost when we consider norms of higher derivatives of $du$ as the connection form $A$ appears when relating the intrinsic and extrinsic representations, and $A$, in general, cannot be controlled. In the case of the Coulomb frame, we can use the smallness provided by Proposition ~\ref{main elliptic estimates} to prove the desired equivalence of Sobolev norms.  To see this,  let $\psi$ be a section of $u^*TN$  whose components in terms of the Coulomb frame $e$ are given by

\begin{align}\label{psi=Q}
\psi= Q^a e_a= Qe
\end{align}
By the previous discussion we have $\|\psi\|_{L^2} = \|Q\|_{L^2}$.  Recall that we can represent covariant derivatives of $\psi$ in terms of the extrinsic partial derivatives of $\psi$ and the second fundamental form  by
 
 \begin{align}\label{Gauss}
 \p_k \psi = D_k \psi + S(u)(\p_k u, \psi)
 \end{align}
 Using the representation \eqref{psi=Q} we also have 

 \begin{align}\label{DQ}
 D_{k} \psi = (\p_k Q + AQ)e
 \end{align}
Combining \eqref{Gauss} and \eqref{DQ}, we obtain

\begin{align*}
 \p_k \psi= \p_k Q e + AQe + B(u)(\p_ku, Qe)
 \end{align*}
 We can then use Proposition ~\ref{main elliptic estimates} ~$(i)$, Sobolev embedding and the boundedness of the second fundamental form to obtain
 
 \begin{align}\label{norm equiv}
 \abs{ \|\p \psi\|_{L^2} -\|\p Q\|_{L^2}} &\lesssim \|AQ\|_{L^2} + \|duQ\|_{L^2} \\ \notag
 &\lesssim (\|A\|_{L^4} + \|du\|_{L^4}) \|Q\|_{L^4}\\ \notag
 &\lesssim \ep \|\p Q\|_{L^2}
 \end{align}
 This proves equivalence of the $H^1$ norms of $Q$ and $\psi$. Interpolation then provides equivalence for the $H^s$ norms for all $0\le s\le 1$. To conclude the equivalence of all the $H^s$ norms of $q$ and $du$, we apply the above argument to $\psi=\n^{\ell} du$ for all $\ell\in \N$. 
 
 Note that a similar argument also proves the equivalence of the $H^s$ norms of $\psi$ if we instead used covariant derivatives on $u^*TN$. That is, we can also show that 
 
 \begin{align}\label{covariant sobolev}
 \|DQ\|_{L^2} \simeq \|\p Q\|_{L^2} = \|Q\|_{\dot{H}^1_i(\R^4;N)} \simeq \|Q\|_{\dot{H}^1_i(M;N)}
 \end{align}
 We will use \eqref{covariant sobolev} in Section \ref{Higher Regularity} when we prove that higher regularity of wave maps is preserved by the evolution.

\vspace{\baselineskip}

\section{Wave Equation for $du$} \label{wave equation for du}
In this section show that for any Riemannian manifold $(M,g)$, if ~$u: \R \times M \to N$ is a smooth wave map, then we can derive wave equations of $1$-forms for $du$. The wave equations of $1$-forms imply  a system of variable coefficient wave equations for the components of $du$. We emphasize that the content of this section holds for any Riemannian manifold $(M,g)$ and not just the special case $(M,g)= (\R^4, g)$ with $g$ as in ~\eqref{g}--\eqref{p^k g}. 

We begin by expressing $du\in\Gamma(T^*\ti{M} \otimes u^*TN)$ in terms of the Coulomb frame $e$ as in Proposition ~\ref{main elliptic estimates}, by  finding $u^*TN$-valued one-forms $q=q_{\al} dx^{\al}$ so that

\begin{align}\label{du=q}
 du = q^{a}e_a
 \end{align}
Here $q^a_{\al} = u^*h\left(\p_{\al}u, e_a\right)$. Assuming that $u$ is a wave map, we derive a wave equation of $1$-forms for $q$. In what follows we let $\Box= d \de + \de d$ denote the Hodge Laplacian on $p$-forms over $\ti{M}= \R\times M$,  where   $d$ is the exterior derivative on $\ti{M}$ and $\de$ is the adjoint to $d$. 
 
\begin{lem}\label{delta q} Let $u: (\tilde{M}, \eta) \to (N, h)$ be a smooth wave map. And let $q=du$ be the representation of $du$ in the Coulomb frame $e$ as in \eqref{du=q}. Then we have
$\de \, q^c = A^c_{a, \al}\, \eta^{\al \be}\, q^a_\be  $.
\end{lem}

\begin{proof} This follows from the fact that $u$ is a wave map. We have that 
u is wave map if and only if

\begin{align*}
\frac{1}{\sqrt{\abs{\eta}}}D_{\al}\left( \sqrt{\abs{\eta}}\,  \eta^{\al \be} \p_{\be}u\right) =0\quad
\Longleftrightarrow \quad
\frac{1}{\sqrt{\abs{\eta}}}D_{\al}\left( \sqrt{\abs{\eta}}\,  \eta^{\al \be}q^a_\be e_a\right) =0
\end{align*} 
\\
Hence, we have 

\begin{align*}
0=&\frac{1}{\sqrt{\abs{\eta}}}D_{\al}\left( \sqrt{\abs{\eta}}\, \eta^{\al \be}q^a_\be e_a\right)\\
=&\frac{1}{\sqrt{\abs{\eta}}}\p_{\al}\left( \sqrt{\abs{\eta}}\, \eta^{\al \be}q^c_\be\right)e_c + \eta^{\al \be} q_{\be}^a A^c_{a, \al} e_c\\
=& -\de q^c e_c+ A^c_{a, \al} \eta^{\al \be}\, q^a_\be \, e_c
\end{align*} 
Therefore,
\begin{align*}
\de \, q^c = A^c_{a, \al}\, \eta^{\al \be}\, q^a_\be 
\end{align*}
as desired. 
\end{proof}
 
\vspace{\baselineskip}

\begin{lem} \label{dq} Let $u: (\tilde{M}, \eta) \to (N,h)$ be a smooth map and let $q=du$ be the representation of $du$ in the Coulomb frame as in \eqref{du=q}. Then we have
$dq^c= -A^c_b \wedge q^b$.
\end{lem}
 
 \begin{proof}
 First we claim that  $D_{\al}(\p_{\be}u) - D_{\be}(\p_{\al} u) =0$. To see this recall that
 
 \begin{align*}
 D_{\al}(\p_{\be}u)^k = \p_{\al}\p_{\be} u^k + \Gamma^k_{i j}(u) \p_{\al}u^j \p_{\be}u^i
 \end{align*}
 \\
 Then the claim follows from the fact that $\p_{\al}\p_{\be} u^k = \p_{\be}\p_{\al} u^k$ and the fact that $\Gamma^{k}_{i j} = \Gamma^k_{j i}$. The above implies that
 
 \begin{align*}
 D_{\al}(q_{\be}^a \, e_a) - D_{\be}(q_{\al}^a \, e_a) =0
 \end{align*}
 \\
 Now, recalling that the $A$ is the connection form for the frame $e$ we have that 
 
\begin{align*}
0& = D_{\al}(q_{\be}^b \, e_b) - D_{\be}(q_{\al}^b \, e_b) \\
&= \left(\p_{\al}q_{\be}^c + A^c_{b, \al} \,q^b_{\be} - \p_{\be}q_{\al}^c - A^c_{b, \be} \,q^b_{\al}\right) e_c\\
& = \left(\p_{\al}q_{\be}^c - \p_{\be}q_{\al}^c + A^c_{b, \al}\, q^b_{\be}  - A^c_{b, \be}\, q^b_{\al}\right) e_c\\
&=\left( (dq^c)_{\al \be}- (A^c_b \wedge q^b)_{\be \al} \right) e_c
\end{align*}
 and the lemma follows.
 \end{proof}
 
Lemma ~\ref{dq} shows that in local coordinates on $\ti{M}$ we have that $(dq^c)_{\al \be}\, dx^{\al}\wedge dx^{\be} = (A^c_b \wedge q^b)_{\be \al} \,  dx^{\al}\wedge dx^{\be}$.  We can abbreviate this by writing $dq = -A\wedge q$.  Hence, by Lemma ~\ref{delta q} and Lemma ~\ref{dq} we obtain the following equation for  ~$q$ 
 
 \begin{align}\label{eqn for q}
 \Box q = d(\eta^{\al \be} A_{\al} q_{\be}) + \de( -A\wedge q)
 \end{align}
This is a system of wave equations for the $u^*TN$ valued $1$-form $q$. In coordinates we can express the operator $\de$ on a $2$-form $\omega$ in terms of  Levi-Civita connection, $\n$, on $\R \times M$ as follows

\begin{align}
(\de \omega)_{ \be} =- ( \n^{\al} \omega )_{\al \be}
\end{align}
where $\n^{\al} = \eta^{\al \be} \n_{\be}$. Hence, in  components, the equations for ~$q$ become
 
 \begin{align}\label{q components}
( \Box q)_{\ga} = \p_{\ga}(\eta^{\al \be} A_{\al} q_{\be}) + \n^{\al}( A\wedge q)_{\al \ga}
\end{align}
By expanding the right-hand side of ~\eqref{q components} we obtain the following equation for $q$.

\begin{prop}\label{wave q} Let $u:(\ti{M}, \eta) \to (N,h)$ be a smooth wave map. Let $q=du$ be the representation of $du$ in the Coulomb frame, $e$ as in \eqref{du=q}. Then $q$ satisfies the following wave equation of $1$-forms,  written in components as

\begin{align}\label{eqn for q components}
(\Box q)_{\ga} =  F_{\ga \al}\,  q^{\al} + A_{\al}\, A^{\al} \, q_{\ga}+ (\n^{\al}A)_{\al}\, q_{\ga} + 2 A^{\al}\, (\n_{\al} q)_{\ga} 
\end{align}
where $F$ is the curvature tensor on $u^*TN$. 
\end{prop}
 
 \begin{rem} Proposition \ref{wave q} essentially amounts to differentiating the wave map equation \eqref{wm} and then expressing the result in terms of the Coulomb frame.  We emphasize that in order to obtain \eqref{eqn for q components} we must begin with a wave map $u$. 
 \end{rem}
\begin{proof}
We begin by expanding the right-hand side of \eqref{q components}
 \begin{align*}
( \Box q)_{\ga} &=  \, \p_{\ga}(\eta^{\al \be} A_{\al} q_{\be}) + \n^{\al}( A\wedge q)_{\al \ga}\\
\\
&=  \, (\p_{\ga}\eta^{\al \be}) A_{\al} q_{\be} + \eta^{\al \be}  (\p_{\ga}A_{\al} ) q_{\be} + \eta^{\al \be} A_{\al}(\p_{\ga} q_{\be}) +  \eta^{\al \be}\left( \n_{\be}A \wedge q\right)_{\al \ga}\\
&\quad + \eta^{\al \be}\left(A\wedge \n_{\be}q\right)_{\al \ga}\\
\\
&=   (\p_{\ga}\eta^{\al \be}) A_{\al} q_{\be} + \eta^{\al \be}  (\p_{\ga}A_{\al} ) q_{\be} + \eta^{\al \be} A_{\al}(\p_{\ga} q_{\be}) + \eta^{\al \be} (\n_{\be}A)_{\al}\, q_{\ga} \\
&\quad -  \eta^{\al \be} (\n_{\be}A)_{\ga}\, q_{\al}+ \eta^{\al \be} A_{\al} (\n_{\be}q)_{\ga}  - \eta^{\al \be} A_{\ga} (\n_{\be}q)_{\al} 
\end{align*}
\\
Now, observe that 

\begin{align*}
\eta^{\al \be} (\p_{\ga} A_{\al}) q_{\be} = \p_{\al} A_{\ga} q^{\al} 
+ F_{\ga \al} q^{\al} - A_{\ga} A_{\al} q^{\al} +  A_{\al} A_{\ga} q^{\al}
\end{align*}
\\
and by Lemma ~\ref{delta q} we have 

\begin{align*} 
-\eta^{\al \be} A_{\ga} (\n_{\be} q)_{\al}= -A_{\ga} (\n^{\al}q)_{\al} = A_{\ga}\de q = A_{\ga} A_{\al} q^{\al}
\end{align*}
\\
Also, Lemma ~\ref{dq} implies

\begin{align*}
\eta^{\al \be} A_{\al} (\p_{\ga}q_{\be}) &= \eta^{\al \be} A_{\al}(dq)_{\ga \be} + \eta^{\al \be} A_{\al}\p_{\be} q_{\ga}\\
\\
&= \eta^{\al \be} A_{\al}(A\wedge q)_{\be \ga} +\eta^{\al \be} A_{\al} \p_{\be}q_{\ga}\\
\\
&=  A_{\al}A^{\al}q_{\ga} -  A_{\al}A_{\ga}q^{\al} + \eta^{\al \be} A_{\al}\p_{\be}q_{\ga}
\end{align*}
\\
Hence, 

 \begin{align*}
( \Box q)_{\ga} &= 
 F_{\ga \al} q^{\al}  +  A_{\al}A^{\al}q_{\ga} +  (\n^{\al}A)_{\al}\, q_{\ga}+ A^{\al} (\n_{\al}q)_{\ga} + \eta^{\al \be} A_{\al}\p_{\be}q_{\ga} \\
 &\quad + (\p_{\ga}\eta^{\al \be}) A_{\al} q_{\be} +\eta^{\al \be}\p_{\be} A_{\ga} q_{\al} -   \eta^{\al \be} (\n_{\be}A)_{\ga}\, q_{\al}  
\end{align*}
\\
Next observe that
\begin{align*}
\eta^{\al \be} A_{\al} \p_{\be} q_{\ga} = \eta^{\al \be} A_{\al}(\n_{\be} q)_{\ga} + \eta^{\al \be} A_{\al}\Gamma^{\sigma}_{\be \ga} q_{\sigma}
\end{align*}
and 
\begin{align*}
\eta^{\al \be} \p_{\be} A_{\ga} q_{\al} - \eta^{\al \be} (\n_{\be} A)_{\ga} q_{\al} = \eta^{\al \be} \Gamma^{\sigma}_{\be \ga} A_{\sigma} q_{\al}
\end{align*}
 Therefore,
  \begin{align*}
( \Box q)_{\ga} &= 
 F_{\ga \al} q^{\al}  +  A_{\al}A^{\al}q_{\ga} +  (\n^{\al}A)_{\al}\, q_{\ga}+ 2A^{\al} (\n_{\al}q)_{\ga}\\
 &\quad + (\p_{\ga}\eta^{\al \be}) A_{\al} q_{\be} + \eta^{\al \be} A_{\al}\Gamma^{\sigma}_{\be \ga} q_{\sigma}   + \eta^{\al \be} \Gamma^{\sigma}_{\be \ga} A_{\sigma} q_{\al}
\end{align*}

Finally, we claim that 

\begin{align}\label{Gamma}
(\p_{\ga}\eta^{\al \be}) A_{\al} q_{\be} + \eta^{\al \be} A_{\al}\Gamma^{\sigma}_{\be \ga} q_{\sigma}   + \eta^{\al \be} \Gamma^{\sigma}_{\be \ga} A_{\sigma} q_{\al}=0
\end{align}
This follows from the fact that $\Gamma^{\sigma}_{\be \ga} = \frac{1}{2}\eta^{\sigma \de}( \p_{\be}\eta_{\ga \de} +  \p_{\ga}\eta_{\be \de}-  \p_{\de}\eta_{\be \ga})$  and that $\p_{\ga}\eta^{\al \be} = -\eta^{\al \de} (\p_{\ga}\eta_{\de \sigma} )\eta^{\sigma \be}$, the latter statement being the general fact that for an invertible matrix $G(x)$ we have $\p_i G^{-1} = -G^{-1} \p_i G G^{-1}$.  To show \eqref{Gamma}, we write

\begin{align}
\notag(\p_{\ga}\eta^{\al \be}) A_{\al} q_{\be} + \eta^{\al \be} A_{\al}\Gamma^{\sigma}_{\be \ga} q_{\sigma}   &+ \eta^{\al \be} \Gamma^{\sigma}_{\be \ga} A_{\sigma} q_{\al}\\ \notag
\\ \notag
&= \left( \p_{\ga} \eta^{\al \be} + \eta^{\al \sigma} \Gamma^{\be}_{\sigma \ga} + \eta^{\sigma \be} \Gamma^{\al}_{\sigma \ga} \right) A_{\al} q_{\be}\\ \notag
\\ \notag
&=\left( -\eta^{\al \de}\eta^{\sigma \be}\p_{\ga} \eta_{\de \sigma}\right) A_{\al \be}\\  \label{sigma de}
&\quad   + \left(\frac{1}{2}\eta^{\al \sigma} \eta^{\be \de}(\p_{\sigma} g_{\de \ga} + \p_{\ga} g_{\sigma \de} - \p_{\de} g_{\sigma \ga})\right) A_{\al} q_{\be}\\ \notag
&\quad + \left(\frac{1}{2}\eta^{\sigma \be}\eta^{\al \de}( \p_{\sigma} g_{\de \ga} + \p_{\ga} g_{\sigma \de} - \p_{\de} g_{\sigma \ga}) \right) A_{\al} q_{\be}\\ \notag
&= 0
\end{align}
where the last line follows by swapping $\sigma$ and $\delta$ in line \eqref{sigma de} above. Therefore,

\begin{align*}
( \Box q)_{\ga} &= 
 F_{\ga \al} q^{\al}  +  A_{\al}A^{\al}q_{\ga} +  (\n^{\al}A)_{\al}\, q_{\ga}+ 2A^{\al} (\n_{\al}q)_{\ga}
 \end{align*}
as claimed.
\end{proof}

Next, we examine the left hand side of ~\eqref{eqn for q}. We claim that for a $1$-form q, we can write $\Box q = \ddot{q} + \Delta q$, where $\Delta$ denotes the Hodge Laplacian on the Riemannian manifold $M$ and $\ddot{q}$ is the $1$-form given in local coordinates by $\ddot{q}(t, x) =  \ddot{q}_{\al}(t,x) \, dx^{\al}$.  
 
 \begin{lem}We can express $\Box$  in local coordinates on $\R\times M$ by
 
 \begin{align}\label{Hodge}
 (\Box q)_{\ga} = \ddot{q}_{\ga} + (\Delta q)_{\ga}
 \end{align}
 \\
 where here $\Delta$ is the Hodge Laplacian on $1$-forms over $M$. 
 
 \end{lem}
 \begin{rem} Despite the appearance of the $+$ sign in expression \eqref{Hodge}, the expression $\Box  = \partial_t^2 + \Delta $ is, in fact, a hyperbolic operator as we will see in Proposition \ref{main eqn for q}. The sign in \eqref{Hodge} is simply due to our sign convention for the Hodge Laplacian $\Delta$. Our convention is such that for a $0$-form $f$, $\Delta f = -\Delta_g f$ where $\ds{\Delta_g = \frac{1}{\sqrt{\abs{g}}}\partial_i( \sqrt{\abs{g}} g^{ij} \partial_j)}$ is the Laplace-Beltami operator on $M$. 
   \end{rem}
 \begin{proof} Let $\eta= \textrm{diag}(-1, g)$ denote the metric on $\ti{M} = \R\times M$. In the following argument,  $0\le \al, \be, \ga \le d$ will be indices denoting coordinates on $\R\times M$  and $1\le i,j \le d$ will be indices denoting coordinates on $M$. Also we denote by ${d}_M$, (resp. $\de_M$), the exterior differential, (resp. co-differential),  on $M$. It follows that
 
 \begin{align*}
 (\Box{q})_{\ga} &= (d\de q)_{\ga} + (\de d q)_{\ga}\\
 \\
 &= -\p_{\ga}\left( \eta^{\al \be} (\n_{\be}q)_{\al}\right) - \eta^{\al \be} (\n_{\be} dq)_{\al \ga}\\
 \\
 &= -\p_{\ga}\left( - (\n_0 q)_0 +  g^{i j} (\n_{j}q)_{i}\right) + \left( \n_0 dq\right)_{0 \ga} - g^{ij}(\n_{j} dq)_{i \ga}\\
 \\
 &=  \p_{\ga} \p_0 q_0 -\p_{\ga}\left( g^{i j} (\n_{j}q)_{i}\right) + \p_0\p_0 q_{\ga}  - \p_{\ga} \p_0 q_0\ - g^{ij}(\n_{j} dq)_{i \ga}\\
 \\
 &= \ddot{q}_{\ga} +(d_M\de_M q)_{\ga} + (\de_M d_M q)_{\ga}\\
 \\
 &= \ddot{q}_{\ga} + (\Delta q)_{\ga}
\end{align*}
 \\
 Above we have used the fact that the Christoffel symbols $\Gamma_{\al \be}^{\de}= 0$ if either $\al, \be$, or $\de$ are equal to $0$. 
\end{proof}
 
We can derive a coordinate representation for the Hodge Laplacian $\Delta$ on $1$-forms in terms of the Laplace-Beltrami  operator, $\Delta_g$,  on functions plus lower order terms.

\begin{lem} The Hodge Laplacian on $1$-forms $q$ can be written in coordinates as
\begin{align}\label{Laplace Beltrami}
(\Delta q)_{\ga}= -\Delta_g q_{\ga} + 2g^{ij}\Gamma^{k}_{j \ga} \p_{i}q_{k}+\p_{\ga}(g^{ij}\Gamma^k_{ij})q_k 
\end{align}
\end{lem}

\begin{proof} Here we will let $d$ and $\de$ denote the exterior differential and exterior co-differential on $M$. Then, 

\begin{align*}
(\Delta q)_{\ga}& = (d\de q)_{\ga} + (\de d q)_{\ga}\\
\\
&= -\p_{\ga}\left( g^{ij}(\n_j q)_i\right) - g^{ij} \left( \n_j dq\right)_{i\ga}\\
\\
&= -(\p_{\ga} g^{ij}) (\n_j q)_i - g^{ij} \p_{\ga}( \p_jq_i - \Gamma^k_{ji} q_k)\\
&\quad - g^{ij}\left( \p_j(dq)_{i\ga} - \Gamma^k_{ji}(dq)_{k \ga} - \Gamma^k_{j\ga}(dq)_{ik} \right)\\
\\
&= -(\p_{\ga} g^{ij})\p_j q_i + \p_{\ga} (g^{ij}\Gamma^k_{ji}) q_k  - g^{ij} \p_j \p_i q_{\ga} + g^{ij} \Gamma^k_{ji} \p_k q_{\ga} \\
&\quad+g^{ij} \Gamma^{k}_{j\ga} \p_i q_k - g^{ij} \Gamma^{k}_{j\ga} \p_k q_i
\end{align*}

Recalling that $\Delta_g q_{\ga} = g^{ij} \p_j \p_i q_{\ga} - g^{ij} \Gamma^k_{ji} \p_k q_{\ga}$  and that $\p_{\ga} g^{ij} = - g^{ik}\p_{\ga} g_{km} g^{mj}$ we have then that 

\begin{align*}
(\Delta q)_{\ga}&= -\Delta_g q_{\ga} +  \p_{\ga} (g^{ij}\Gamma^k_{ji}) q_k +  g^{ik}\p_{\ga} g_{km} g^{mj}\p_j q_i + g^{ij} \Gamma^{k}_{j\ga} \p_i q_k - g^{ij} \Gamma^{k}_{j\ga} \p_k q_i
\end{align*}
Finally observe that 

\begin{align*}
 g^{ik}\p_{\ga} g_{km} g^{mj}\p_j q_i + g^{ij} \Gamma^{k}_{j\ga} \p_i q_k - g^{ij} \Gamma^{k}_{j\ga} \p_k q_i = 2g^{ij}\Gamma^{k}_{j \ga} \p_{i}q_{k}
 \end{align*} 
 Therefore 
 
 \begin{align*}
 (\Delta q)_{\ga}&= -\Delta_g q_{\ga} + 2g^{ij}\Gamma^{k}_{j \ga} \p_{i}q_{k}+\p_{\ga}(g^{ij}\Gamma^k_{ij})q_k 
\end{align*}
which is exactly ~\eqref{Laplace Beltrami}.
  \end{proof}
 
\vspace{\baselineskip}

 Combining the results of the previous two lemmas with Proposition \ref{wave q} gives us a system of nonlinear wave equations for the components of $q$. The following Proposition is the main result of this section and will be used to prove a priori estimates for the differential, $du$, of a wave map $u$. 
 
 \begin{prop}\label{main eqn for q} Let $u: (\ti{M}, \eta) \to (N,h)$ be a smooth wave map.  Let $q=du$ be the representation of $du$ in the Coulomb frame, $e$ as in \eqref{du=q}. Then, the components of $q$ satisfy the following system of variable coefficient wave equations: 
 
 \begin{multline}\label{scalar eqn for q} 
 \ddot{q}_{\ga} -\Delta_g q_{\ga} + 2g^{ij}\Gamma^{k}_{j \ga} \p_{i}q_{k}+\p_{\ga}(g^{ij}\Gamma^k_{ij})q_k\\
 =  F_{\ga \al}\,  q^{\al} + A_{\al}\, A^{\al} \, q_{\ga}
 + (\n^{\al}A)_{\al}\, q_{\ga} + 2 A^{\al}\, (\n_{\al} q)_{\ga}  
\end{multline}
Expanding the term $\Delta_{g} q_{\ga}$, the left-hand side of the above system becomes 
\begin{align}\label{div form}
\ddot{q}_{\ga} - g^{ij} \p_i \p_j q_{\ga} +g^{ij} \Gamma^{k}_{ij} \p_k q_{\ga} + 2g^{ij}\Gamma^{k}_{j \ga} \p_{i}q_{k}+\p_{\ga}(g^{ij}\Gamma^k_{ij})q_k 
\end{align}
\end{prop}

\vspace{\baselineskip}

\section{A Priori Estimates} \label{a priori estimates}
To derive a priori bounds for wave maps $u$ we use the Strichartz estimates for variable coefficient wave equations proved in ~\cite{Met-Tat}. We require a Lorentz space refinement of the estimates in ~\cite{Met-Tat} obtained by a rephrasing in terms of Besov spaces and real interpolation. Equation ~\eqref{scalar eqn for q}, the decay assumptions on the metric $g$ specified in ~\eqref{g}--\eqref{p^2 g},  and  ~\cite[Theorems $4$ and $6$]{Met-Tat} imply the following estimates for $q$: 

\begin{align}\label{Besov Strichartz}
\|q\|_{L^2_t \dot{B}^{\frac{1}{6}}_{6,2}} + \| \p q\|_{L^{\infty}_tL^2_x}   \lesssim \|q[0]\|_{\dot{H}^1\times L^2}  + \|H\|_{L^1_t L^2_x}
\end{align}
where $H_{\ga}:= F_{\ga \al}\,  q^{\al} + A_{\al}\, A^{\al} \, q_{\ga}+ (\n^{\al}A)_{\al}\, q_{\ga} + 2 A^{\al}\, (\n_{\al} q)_{\ga} $ is the nonlinearity in ~\eqref{scalar eqn for q}. There are a few things to note. The first is that we have extended the result in ~\cite{Met-Tat} to the case of a system of variable coefficient equations as $q$ is the solution to such a system. However this extension is immediate as the methods in ~\cite{Met-Tat} allow us to treat the lower order terms in ~\eqref{scalar eqn for q} perturbatively, and the principle part of our operator is diagonal. Hence the system of equations for $q$ in ~\eqref{scalar eqn for q} falls directly into the class of equations that are treated in ~\cite{Met-Tat} because of the assumptions in ~\eqref{g}--\eqref{p^2 g}. The second observation is that a Besov norm appears on the left-hand side in ~\eqref{Besov Strichartz}. This refinement can be obtained by an easy modification of the proof of Lemma 19 in ~\cite{Met-Tat}. For completeness we carry out this refinement in Section ~\ref{Strichartz}.

  To obtain a Lorentz space version of estimate ~\eqref{Besov Strichartz} we use the Besov space embedding into Lorentz spaces, see Lemma ~\ref{Sob embedding Lorentz}, with $d=4$, $s=\frac{1}{6}$, $q=6$, $p=8$ and $r=2$ which gives 

\begin{align*}
\dot{B}^{\frac{1}{6}}_{6,2}(\R^4) \hookrightarrow L^{8,2}(\R^4)
\end{align*}
This, together with the estimate in ~\eqref{Besov Strichartz}, gives 

\begin{align}\label{Lorentz Strichartz} 
\|q\|_{L^2_t L^{8,2}_x} + \| \p q\|_{L^{\infty}_tL^2_x}  \lesssim \|q[0]\|_{\dot{H}^1\times L^2} +  \|H\|_{L^1_t L^2_x}
\end{align}

We use Proposition ~\ref{main elliptic estimates}, together with Sobolev embedding to estimate the various terms in $H$. In local coordinates on $M$, $H$ is given by 

\begin{align}\label{H}
H_{\ga} &=\eta^{\ga \al} F_{\ga \be}\,  q_{\al} + \eta^{\al \be} A_{\al}\, A_{\be} \, q_{\ga}+ \eta^{\al \be} (\p_{\be}A_{\al})\, q_{\ga} \\
&\quad + 2\eta^{\al \be} A_{\be} \Gamma^{\de}_{\al \ga} q_{\de} + 2 \eta^{\al \be}A_{\be}\, (\p_{\al} q_{\ga}) \notag
\end{align}
Hence, at any time $t \in [0, T)$, (where $T $ is chosen as in ~\eqref{bootstrap}, for the sake of our bootstrap argument), we have

\begin{align}\label{L^2 H}
\|H\|_{L^2_x} &\lesssim \|Fq\|_{L^2_x} + \|A^2 q\|_{L^2_x} + \|(\p A) q\|_{L^2_x} + \|\Gamma A q\|_{L^2_x} + \|A \p q\|_{L^2_x}\\ \notag
\\ \notag
&\lesssim \|F\|_{L^4_x} \|q\|_{L^4_x} + \|A^2\|_{L_x^{\frac{8}{3}}}\|q\|_{L_x^8} + \|\p A\|_{L_x^{\frac{8}{3}}}\|q\|_{L_x^8} \\
&\quad+ \|\Gamma\|_{L^4_x} \|A\|_{L_x^8} \|q\|_{L_x^8} + \|A\|_{L_x^{\infty}} \|\p q\|_{L_x^2}\notag \\ \notag 
\\ \notag 
&\lesssim  \| q\|_{L_x^8}^2 \|\p q\|_{L_x^2} + \|A\|_{\dot{W}_x^{1, \frac{8}{3}}}\|q\|_{L_x^8} + \|q\|_{L_x^{8,2}}^2\|\p q\|_{L_x^2}\\ \notag
\\ \notag 
&\lesssim \|q\|_{L_x^{8,2}}^2 \|\p q\|_{L_x^2}  
\end{align}
where in the third inequality above we have used  Proposition ~\ref{main elliptic estimates} and Sobolev embedding to show that 
$\|A^2\|_{L^{\frac{8}{3}}} \lesssim \|A\|_{L^4} \|A\|_{L^8} \lesssim \|A\|_{\dot{W}^{1, \frac{8}{3}}}$. This implies that we have the estimate 

\begin{align}\label{q local}
\|q\|_{L^2_t([0,T); L_x^{8,2})} + \| \p q\|_{L^{\infty}_t([0,T);L_x^2)}   &\lesssim \|q[0]\|_{\dot{H}^1\times L^2}\\ \notag
&\quad +  \| q\|_{L_t^2([0,T); L_x^{(8,2)})}^2 \|\p q\|_{L_t^{\infty}([0,T);L^2_x)}  \\ \notag
\\\notag
&\lesssim \|q[0]\|_{\dot{H}^1\times L^2}\\ \notag
&\quad + \left(\|q\|_{L^2_t([0,T); L_x^{(8,2)})} + \| \p q\|_{L_t^{\infty}([0,T); L_x^2)}\right)^3
\end{align}
 By the equivalence of the extrinsic and intrinsic norms of $du=q$,  see Section ~\ref{Equivalence of Norms}, we can show 
 
 \begin{align}
 \|q[0]\|_{\dot{H}^1 \times L^2} \lesssim \|du_0\|_{\dot{H}^1} + \|u_1\|_{\dot{H}^1} \lesssim \ep_0
 \end{align}
  Hence as long as $\ep_0$ is sufficiently small, we can use a bootstrap/continuity-trapping argument to absorb  the cubic term,
  \begin{align*}
  \left(\|q\|_{L^2([0,T); L^{(8,2)})} + \| \p q\|_{L^{\infty}([0,T); L^2)}\right)^3
  \end{align*}
   on the left-hand side in \eqref{q local} and obtain the global in time estimate
 
 \begin{align} \label{q global}
 \|q\|_{L^2 L^{8,2}} + \| \p q\|_{L^{\infty}L^2}   &\lesssim  \|du_0\|_{\dot{H}^1} + \|u_1\|_{\dot{H}^1}
 \end{align}

Using again the equivalence of the relevant extrinsic norms of $du$ and intrinsic norms of $q$, see Section ~\ref{Equivalence of Norms}, and recalling that $\|du\|_{L^2L^8} \le \|du\|_{L^2L^{8,2}}$, we obtain the desired global a priori bounds which we record in the following proposition:

\begin{prop}\label{global est}
Let $(\ti{M}, \eta) =(\R \times \R^4, \eta)$, where $\eta = \mathrm{diag}(-1, g)$, and $g$ satisfies the conditions ~\eqref{g}--\eqref{p^2 g}. Let $u: (\ti{M}, \eta) \to (N, h)$ be a smooth wave map with initial data $(u_0, u_1)$ satisfying ~\eqref{small energy}. Then $du$ satisfies the following global, a priori estimates
\begin{align}\label{du global}
\|du\|_{L^2_t L^{8}_x} + \|du\|_{L_t^{\infty}\dot{H}^1_x}  \lesssim \|du_0\|_{\dot{H}^1} + \|u_1\|_{\dot{H}^1} \lesssim \ep_0
 \end{align}
 \end{prop}

\section{Higher Regularity}\label{Higher Regularity}

In this section we show that higher regularity of the data is preserved. In particular, we show that if we begin with initial data,  $(u_0, u_1) \in H^s\times H^{s-1}((\R^4,g), TN)$ for any $s \ge 2$, such that ~\eqref{small energy} holds, then the $H^s\times H^{s-1}$ norm of the solution, $(u(t), \dot{u}(t))$, to ~\eqref{wm}, is finite for any time $t$.   This will allow us to immediately deduce global existence of wave maps with data $(u_0, u_1) \in H^{s} \times H^{s-1}$ satisfying ~\eqref{small energy} for $s\ge 5$, as any local solution to the Cauchy problem can then be extended past any finite time, $T$, using the high regularity local theory with data $(u(T), \dot{u}(T))$, which is finite in $H^{s} \times H^{s-1}$ due to the results in this section. We note that the a priori estimates,  ~\eqref{du global}, and in particular the global control of  $\|du\|_{L^2_tL^8_x}$, will play a key role in the argument. We formulate the main result of this section in the following proposition:

\begin{prop}\label{high reg}
Let $(\ti{M}, \eta) =(\R \times \R^4, \eta)$, where $\eta = \mathrm{diag}(-1, g)$, and $g$ satisfies the conditions ~\eqref{g}--\eqref{p^2 g}. Let $u: (\ti{M}, \eta) \to (N, h)$ be a solution to \eqref{wm} with initial data $(u_0, u_1)$ that is small in the sense of ~\eqref{small energy}. Suppose in addition that  $(u_0, u_1) \in  H^s\times H^{s-1}((\R^4,g), TN)$ with $s \ge 2$. Then for any time $T$, the $H^s\times H^{s-1}((\R^4,g), N)$ norm of the solution $(u(T), \dot{u}(T))$  is finite. In particular,

\begin{align}\label{H^s times H^{s-1}}
\sup_{0\le t\le T}\norm{(u(t), \dot{u}(t))}_{H^s \times H^{s-1}} \le C_T \norm{u[0]}_{H^{s} \times H^{s-1}}
\end{align}  
where the constant, $C_T$, depends on $T$ and $\ep_0$. 
\end{prop}

To prove Proposition \ref{high reg}, we begin by differentiating ~\eqref{wm} covariantly. 
Let $1\le \ga \le 4$ be a space index and let $q=du$ be the representation of $du$ in the Coulomb frame. Then, recalling that $D_{\al}D_{\be} - D_{\be}D_{\al} =F_{\al \be}$, we have 

\begin{align*} 
0= &\, D_{\ga}\left( \frac{1}{\sqrt{\abs{\eta}}} D_{\al}( \sqrt{\abs{\eta}} \eta^{\al \be} q_{\be})\right)\\
\\
= &-D_{\ga}\left( D_t q_t\right) + D_{\ga}\left( \frac{1}{\sqrt{\abs{g}}} D_{\al}( \sqrt{\abs{g}} g^{\al \be} q_{\be} )\right)\\
\\
=& -D_tD_{t}  q_{\ga}  +    \frac{1}{\sqrt{\abs{g}}} D_{\al} (\sqrt{\abs{g}} g^{\al \be} D_{\ga}q_{\be} )\\
\\
&+ F_{\ga \al} \eta^{\al \be} q_{\be}  + \p_{\ga}(g^{\al \be})D_{\al} q_{\be}+  
\p_{\ga}\left(\frac{1}{\sqrt{\abs{g}}} \p_{\al}(\sqrt{\abs{g}} g^{\al \be})\right) q_{\be}
\end{align*}
This implies  that $q$ satisfies the equation

\begin{multline}\label{1 derivative}
D_tD_tq_{\ga} -  \frac{1}{\sqrt{\abs{g}}} D_{\al} (\sqrt{\abs{g}} g^{\al \be} D_{\be}q_{\ga} )
\\  
= F_{\ga \al} \eta^{\al \be} q_{\be}  + \p_{\ga}(g^{\al \be})D_{\al} q_{\be} -  \p_{\ga}\left(g^{\al \rho} \Gamma^{\be}_{\al \rho}\right) q_{\be}
\end{multline}

Pairing this equation with $g^{\ga \de}D_t q_{\de}$ as sections of $u^*TN\to M$ and integrating over $M$ gives 

\begin{align*}
& \int_M \ang{  D_tD_tq_{\ga} -  \frac{1}{\sqrt{\abs{g}}} D_{\al} (\sqrt{\abs{g}} g^{\al \be} D_{\ga}q_{\be} ), \, g^{\ga \de}D_t q_{\de}} \sqrt{\abs{g}} \, dx 
\\
&=\,  \int_M \ang{  F_{\ga \al} \eta^{\al \be} q_{\be}, \, g^{\ga \de}D_t q_{\de}} \sqrt{\abs{g}} \, dx + \int_M \ang{ \p_{\ga}(g^{\al \be})D_{\al} q_{\be} , \, g^{\ga \de}D_t q_{\de}} \sqrt{\abs{g}} \, dx \\
&\quad- \int_M \ang{  \p_{\ga}\left(g^{\al\rho} \Gamma^{\be}_{\al \rho}\right) q_{\be}, \, g^{\ga \de}D_t q_{\de}} \sqrt{\abs{g}} \, dx 
\end{align*}
Integrating the second term on the left by parts gives 

\begin{align*}
\frac{1}{2} \frac{d}{dt} \|Dq\|_{L^2}^2 &= - \int_M g^{\al \be}\p_{\al}(g^{\ga \de}) \ang{    D_{\ga}q_{\be}  , \, D_t q_{\de}} \sqrt{\abs{g}} \, dx\\
 &\quad+ \int_M g^{\al \be}g^{\ga \de}\ang{    D_{\ga}q_{\be}  , \, F_{\al t} q_{\de}} \sqrt{\abs{g}} \, dx \\
 &\quad+ \int_M \eta^{\al \be} g^{\ga \de}\ang{  F_{\ga \al} q_{\be}, \, D_t q_{\de}} \sqrt{\abs{g}} \, dx\\ 
 &\quad + \int_M  \p_{\ga}(g^{\al \be})g^{\ga \de}\ang{D_{\al} q_{\be} , \, D_t q_{\de}} \sqrt{\abs{g}} \, dx \\
&\quad- \int_M \p_{\ga}\left(g^{\al\rho} \Gamma^{\be}_{\al \rho}\right)g^{\ga \de}\ang{   q_{\be}, \, D_t q_{\de}} \sqrt{\abs{g}} \, dx 
\end{align*}
where we define
\begin{align}
\|Dq\|_{L^2}^2 &:=  \int_M g^{\ga \de} \ang{ D_{\ga}q_t, \, D_{\de}q_{t}} \sqrt{\abs{g}} \, dx\\
&\quad + \int_M g^{\al \be}g^{\ga \de}\ang{    D_{\ga}q_{\be}  , \, D_{\al}q_{\de}} \sqrt{\abs{g}} \, dx \notag
\end{align}
 Hence,

\begin{align}\label{dt Dq}
\frac{1}{2} \frac{d}{dt} \|Dq\|_{L^2}^2 &\lesssim  \|\p g\|_{L^{\infty}} \|Dq\|_{L^2}^2 + \|F\|_{L^4}\|q\|_{L^4} \|Dq\|_{L^2} \\ \notag
&\quad + \|\p^2 g\|_{L^4} \|q\|_{L^4} \|Dq\|_{L^2}\\  \notag
\\ \notag
&\lesssim \|Dq\|_{L^2}^2 \|q\|_{L^8}^2 + \|Dq\|_{L^2}^2
\end{align}
Integrating in time gives  

\begin{align*}
\|Dq(t)\|_{L^2}^2 \le \|Dq(0)\|_{L^2}^2  + C \int_0^t \|Dq(s)\|_{L^2}^2 (\|q(s)\|_{L^8}^2 +1) \, ds
\end{align*}
Hence by Gronwall's inequality we  have 

\begin{align}\label{Dq estimates}
\|Dq(t)\|_{L^2}^2 &\le \|Dq(0)\|_{L^2}^2 \textrm{exp}\left( C \int_0^t (\|q(s)\|_{L^8}^2 +1) \, ds\right)\\ \notag \\ \notag
&\le \|Dq(0)\|_{L^2}^2 \textrm{exp}\left( C(\|q\|_{L^2L^8}^2 +t) \right)\\ \notag \\ \notag
&\le \|Dq(0)\|_{L^2}^2 \textrm{exp}\left( C(\ep_0 +t)\right)
\end{align}
The last inequality follows from the global a priori bounds, ~\eqref{du global}, proved in the previous section. 

As explained in Section ~\ref{Equivalence of Norms}, see   ~\eqref{norm equiv} and ~\eqref{covariant sobolev},  the inequality in ~\eqref{Dq estimates} is equivalent to a bound on $(u, \dot{u})$ in ~$\dot{H}^2(M;N)\times \dot{H}^1(M;N)$. We thus obtain a bound on $(u, \dot{u})$ in $H^2(M;N)\times H^1(M;N)$ by combining the above with the conservation of energy derived in Section ~\ref{Extrinsic Definition}, and the simple $L^2$ estimates obtained by the fundamental theorem of calculus and Minkowksi's inequality

\begin{align}\label{L^2}
\|u(t)\|_{L^2} \le \|u(0)\|_{L^2} + t \|\p_t u(t)\|_{L^2}
\end{align}

\begin{rem}
Of course, we already have proved an even stronger result than ~\eqref{Dq estimates} in the previous section where we showed that, in fact, $\|q(t)\|_{\dot{H^1}} \lesssim \ep_0$ for any time $t$ where the wave map $u$ is defined. We have gone through the trouble in proving ~\eqref{Dq estimates} here in order to establish the technique required to prove bounds on higher derivatives of $q$ below. 
\end{rem}

To obtain bounds in  $H^3(M;N)\times H^2(M;N)$ and in  $H^4(M;N)\times H^3(M;N)$ we proceed in exactly the same manner as above, differentiating ~\eqref{1 derivative} two more times and obtaining wave equations for $D_{\ka} q_{\de}$ and for $D_{\mu}D_{\ka} q_{\de}$. Roughly, these are of the form

\begin{align}\label{2 derivative}
D_t D_t D_{\ka} q_{\ga} - \frac{1}{\sqrt{\abs{g}}} D_{\al}\left( \sqrt{\abs{g}} g^{\al \be} D_{\be} D_{\ka} q_{\ga}\right) &= D_{\ka}(\eta^{\al \be} F_{\ga \al} q_{\be})\\  \notag
&\quad + \textrm{lower order terms}
\end{align}
and
\begin{multline}\label{3 derivative}
D_t D_t D_{\mu}D_{\ka} q_{\ga} - \frac{1}{\sqrt{\abs{g}}} D_{\al}\left( \sqrt{\abs{g}} g^{\al \be} D_{\be} D_{\mu}D_{\ka} q_{\ga}\right) =\\= D_{\mu}D_{\ka}(\eta^{\al \be} F_{\ga \al} q_{\be}) 
+ \textrm{lower order terms} 
 \end{multline}
 Proceeding as above, we pair ~\eqref{2 derivative} with $g^{\ka \io} g^{\ga \de}D_t D_{\io} q_{\de}$, and we pair ~\eqref{3 derivative} with $g^{\mu \nu} g^{\ka \io} g^{\ga \de}D_{t} D_{\nu } D_{\io} q_{\de}$ and integrate over $M$ to obtain for $\ell = 1, 2$
 
 \begin{align}\label{D^2q and D^3q}
 \frac{1}{2}\frac{d}{dt} \|D^{\ell+1}q\|_{L^2}^2 \lesssim \sum_{k=1}^{\ell+1} \|D^kq\|_{L^2}^2 + \|D^{\ell} (\eta F q)\|_{L^2} \|D^{\ell+1} q\|_{L^2}
 \end{align}
We claim that \eqref{D^2q and D^3q} implies the estimate 

 \begin{align}\label{D2 and D3}
  \frac{1}{2}\frac{d}{dt} \|D^{\ell+1}q\|_{L^2}^2 \lesssim \sum_{k=1}^{\ell+1} \|D^kq\|_{L^2}^2 +\left(\sum_{k=1}^{\ell+1} \|D^kq\|_{L^2}^2\right) \|q\|_{L^8}^2 
  \end{align}
In order to deduce \eqref{D2 and D3} from \eqref{D^2q and D^3q}, we need the following lemma:

\begin{lem}\label{D^l F}For any time $t$ and for $\ell=1, 2$ we have 
\begin{align}
\|D^{\ell}(\eta F q)\|_{L^2} \lesssim \sum_{k=1}^{\ell+1} \|D^k q\|_{L^2} \|q\|_{L^8}^2
\end{align}
\end{lem}
\begin{proof}In what follows we will freely use the equivalence of norms explained in Section ~\ref{Equivalence of Norms}. For $\ell=1$, we have $\p(\eta F q) =  \p\eta F q + \eta \p F q + \eta F \p q$. Schematically, recall that we have $F= R(u)(q,q)$ and hence $\p F = (\p R(u))(q,q,q) + 2 R(u)(\p q, q)$. Hence we have 

\begin{align}\label{ p eta F q}
\|\p (\eta F q) \|_{L^2} &\lesssim \|\p \eta\|_{L^4} \|q^3\|_{L^4} + \|\p R(u)\|_{L^{\infty}} \|q\|_{L^4} \|q^3\|_{L^4}\\ \notag
&\quad + \|R(u)\|_{L^{\infty}} \|\p q\|_{L^4} \|q^2\|_{L^4}\\ \notag
\\ \notag
&\lesssim \|q^3\|_{L^4} + \|q\|_{\dot{H}^1} \|q^3\|_{L^4} + \|D^2 q\|_{L^2} \|q\|_{L^8}^2
\end{align}
Finally we claim that $ \|q^3\|_{L^4} \lesssim  \|D^2q\|_{L^2}\|q\|^2_{L^8}$. This follows from the multiplicative Sobolev inequality, see ~\cite[pg. 24]{Fri}. Indeed,

\begin{align}\label{q^3}
\|q^3\|_{L^4} \lesssim \prod_{i=1}^3 \|q\|_{L^{p_i}}& \lesssim  \prod_{i=1}^3 \|D^2q\|_{L^2}^{1-\theta_i} \|q\|_{L^8}^{\theta_i} \lesssim  \|D^2q\|_{L^2}\|q\|^2_{L^8}
\end{align}
as long as we set $\frac{1}{p_1} + \frac{1}{p_2} + \frac{1}{p_3} =\frac{1}{4}$, $\frac{1}{p_i} = \frac{\theta_i}{8}$ and $\theta_1 + \theta_2 + \theta_3 =2$. For example, we can set $p_i= 12$ and $\theta_i= \frac{2}{3}$ for $i= 1, 2, 3$. 

For $\ell=2$ we have 
\begin{align*}
\p^2(\eta F q) = \p^2 \eta Fq + \eta \p^2 F q + \eta F \p^2 q + 2 \p \eta  \p F q + 2 \p \eta  F \p q +  2 \eta \p F \p q
\end{align*}
And we have 
\begin{align*}
\p^2 F = (\p^2R(u)) (q,q,q,q) + 5(\p R(u))(\p q, q, q) + 2 R(u)(\p^2 q, q) + 2 R(u)(\p q, \p q)
\end{align*} 
Hence,

\begin{align}
\| \p^2( \eta F q)\|_L^2  &\lesssim \|\p^2 \eta Fq\|_{L^2} + \| \p \eta  \p F q\|_{L^2} + \| \p \eta  F \p q\|_{L^2}\label{p eta}\\\notag
\\
&\quad +\| \eta F \p^2 q\|_{L^2} +  \| \eta \p F \p q\|_{L^2} + \|\eta \p^2 F q\|_{L^2} \label{eta}
\end{align}
The first three terms on the right-hand side of ~\eqref{p eta} all have derivatives hitting $\eta$ and can be controlled as follows

\begin{align*}
\|\p^2 \eta Fq\|_{L^2} + \| \p \eta  \p F q\|_{L^2} + \| \p \eta  F \p q\|_{L^2} &\lesssim \|\p^2 \eta\|_{L^4} \|Fq\|_{L^4} + \|\p \eta\|_{L^{\infty}} \|\p F q\|_{L^{2}}\\
&\quad + \|\p \eta\|_{L^{\infty}} \|F\p q\|_{L^2}\\
\\
&\lesssim \|q^3\|_{L^4} + \|\p R(u)\|_{L^{\infty}} \|q\|_{L^4} \|q^3\|_{L^4} \\
&\quad + \|R(u)\|_{L^{\infty}} \|\p q\|_{L^4} \|q^2\|_{L^4}\\
&\quad + \|\p q\|_{L^4}\|q^2\|_{L^4}\\
\\
&\lesssim \|D^2q\|_{L^2} \|q\|_{L^8}^2
\end{align*}
where the last line is deduced via the same argument as in ~\eqref{ p eta F q} and ~\eqref{q^3}. To estimate the first term in ~\eqref{eta} we observe that

\begin{align*}
\|\eta F \p^2 q\|_{L^2} \lesssim \|F\|_{L^4} \|\p^2 q\|_{L^4} \lesssim \|D^3q\|_{L^2} \|q\|^2_{L^8}
\end{align*}
For the last two terms in ~\eqref{eta} we have 

\begin{align}
 \| \eta \p F \p q\|_{L^2} + \|\eta \p^2 F q\|_{L^2} &\lesssim \|(\p R)(u)\|_{L^{\infty}}\|q^3 \p q\|_{L^2}\\ \notag
 &\quad + \|R(u)\|_{L^{\infty}} \|q (\p q)^2\|_{L^2}\label{p F p^2 F}\\ \notag
 & \quad+ \|R(u)\|_{L^{\infty}} \|q^2 \p^2 q\|_{L^2} \\ \notag
 &\quad+ \|(\p^2 R)(u)\|_ {L^{\infty}} \|q^5\|_{L^2}\\ \notag
  \\ \notag
  &\lesssim \|q\|_{L^4} \|q\|^2_{L^{16}} \|\p q\|_{L^8} + \|q\|_{L^8} \|\p q\|_{L^{\frac{16}{3}}}^2\\
  &\quad + \|q\|_{L^8}^2 \|\p^2 q\|_{L^4} + \|q\|_{L^8}^2 \|q\|_{L^{12}}^3 \notag
 \end{align}
The multiplicative Sobolev inequality then implies 
 
 \begin{align*}
 \|q\|_{16} &\lesssim \|D^3q\|_{L^2}^{\frac{1}{3}} \|q\|_{L^8}^{\frac{2}{3}}\\
 \|\p q\|_{L^8}& \lesssim \|D^3q\|_{L^2}^{\frac{1}{3}} \|q\|_{L^8}^{\frac{2}{3}}\\
  \|\p q\|_{L^{\frac{16}{3}}} &\lesssim \|D^3q\|_{L^2}^{\frac{1}{2}} \|q\|_{L^8}^{\frac{1}{2}} \\
 \| q\|_{L^{12}} &\lesssim \|D^3q\|_{L^2}^{\frac{1}{3}} \|q\|_{L^4}^{\frac{2}{3}}\\
 \end{align*}
 Plugging these into ~\eqref{p F p^2 F}, and using Sobolev embedding followed by the a priori bounds  $\| q\|_{L^{\infty} \dot{H}^1} \lesssim \ep_0$, we get

 \begin{align*}
  \| \eta \p F \p q\|_{L^2} + \|\eta \p^2 F q\|_{L^2} &\lesssim \|q\|_{\dot{H}^1}\|D^3q\|_{L^2} \|q\|_{L^8}^2
  \\
  \\
  &\lesssim \|D^3q\|_{L^2} \|q\|_{L^8}^2
 \end{align*}
  Putting this all together we conclude
  
  \begin{align*}
  \|\p^2(\eta F q)\|_{L^2} \lesssim \left(\|D^2 q\|_{L^2} + \|D^3 q\|_{L^2}\right) \|q\|_{L^8}^2
  \end{align*}
  as desired. 
  \end{proof}

 Now, inserting the conclusion in  Lemma ~\ref{D^l F} into ~\eqref{D^2q and D^3q} we have  for $\ell=1,2$
 
 \begin{align}
  \frac{1}{2}\frac{d}{dt} \|D^{\ell+1}q\|_{L^2}^2 \lesssim \left(\sum_{k=1}^{\ell+1} \|D^kq\|_{L^2}^2 \right)\left(\|q\|_{L^8}^2 + 1 \right) 
  \end{align}
Together with ~\eqref{dt Dq} this implies for $\ell=1, 2$ that 

\begin{align}\label{dt D^2 and D^3}
\frac{1}{2}\frac{d}{dt} \sum_{k=1}^{\ell+1} \|D^{k}q\|_{L^2}^2 \lesssim \left(\sum_{k=1}^{\ell+1} \|D^kq\|_{L^2}^2 \right)\left(\|q\|_{L^8}^2 + 1 \right) 
  \end{align}
Integrating in time, applying Gronwall's inequality and using the a priori estimates $\|q\|_{L^2L^8} \lesssim ~\ep_0$ gives 

\begin{align}\label{D^2 and D^3}
\sum_{k=1}^{\ell+1} \|D^{k}q(t)\|_{L^2}^2 \le \left(\sum_{k=1}^{\ell+1} \|D^{k}q(0)\|_{L^2}^2\right) \textrm{exp}\left( C( \ep_0 + t) \right)
\end{align}
for $\ell=1, 2$. This implies that the $H^3(M;N)\times H^2(M;N)$ (resp.  $H^4(M;N)\times H^3(M;N)$) norm of the solution $(u, \dot{u})$ remains finite for all time assuming the data $(u_0, u_1)$ is bounded in $H^3(M;N)\times H^2(M;N)$, (resp. $H^4(M;N)\times H^3(M;N)$). 

To deal with higher derivatives, $s \ge5$, we note that ~\eqref{D^2 and D^3} implies that $q(t)\in H^3\hookrightarrow  L^{\infty}_x$, and hence we can bootstrap the preceding argument, in particular Lemma ~\ref{D^l F}, to all higher derivatives. For the global existence proof to come in the next section, we only need to do the case $s=5$ as we have a local well-posedness theory for ~\eqref{wm local} at this regularity, see for example \cite[Chapter 5]{Sha-Stru GWE}.

\vspace{\baselineskip}

\section{Existence \& Proof of Theorem \ref{main thm}} \label{Existence}
In this section, we  the complete the proof of  Theorem ~\ref{main thm}. We begin by establishing the existence statement in Theorem \ref{main thm}.  The argument here follows exactly as in ~\cite{Shat-Stru WM}. As explained in Section ~\ref{density}, we can find a sequence of smooth data $ (u_0^k, u_1^k) \in C^{\infty} \times C^{\infty}(M; TN)$ such that $(u_0^k, u_1^k) \to (u_0, u_1)$ in $H^2\times H^1 (M; TN)$ as $k\to \infty$. Using the high regularity, local well-posedness theory, we can find smooth local solutions $u^k$ to the Cauchy problem ~\eqref{wm} with data $(u^k_0, u^k_1)$ satisfying 

\begin{align}\label{smooth data}
\|u^k_0\|_{\dot{H}^2} + \|u_1^k\|_{\dot{H}^1} < \ep_0
\end{align}
for large enough $k$. Then, by the a priori bounds in Proposition ~\ref{global est} and the regularity results in Proposition ~\ref{high reg}, these local solutions $u^k$ can be extended as smooth solutions of ~\eqref{wm} for all time satisfying the uniform in $k$, global-in-time estimates

\begin{align}\label{u^k}
\|du^k\|_{L^{\infty}\dot{H}^1} + \|du^k\|_{L^2L^8} \lesssim \|u_0^k\|_{\dot{H}^2} + \|u_1^k\|_{\dot{H}^1} \lesssim \ep_0
\end{align}
for large enough $k$.  To see this, suppose that the smooth local solution $u^k $ exists on the time interval $[0, T)$. Then, by Proposition ~\ref{global est} and Proposition ~\ref{high reg} we have, say, that the $H^{5} \times H^{4}$ norm of $(u^k(T), \dot{u}^k(T))$ is finite. Hence, we can apply the high regularity local well-posedness theory again to the Cauchy problem with data $(u^k(T), \dot{u}^k(T))$ obtaining a positive time of existence, $T_1$. By the uniqueness theory, this solution agrees with $u^k,$ thereby extending $u^k$ to the interval $[0, T+T_1)$. This implies that $u^k$ is, in fact a global solution, as it can always be extended.

Now, by ~\eqref{u^k}, we can find a subsequence, $u^k$ such that $u^k \rightharpoondown u$ weakly in $H^2_{\textrm{loc}}$. We also have 

\begin{align} \label{global u}
\|du\|_{L^{\infty}\dot{H}^1} + \|du\|_{L^2L^8} \lesssim \|u_0\|_{\dot{H}^2} + \|u_1\|_{\dot{H}^1} \lesssim \ep_0
\end{align}
By Rellich's theorem, we can find a further subsequence so that $du^k \to du$  pointwise almost everywhere. It follows that $u$ is a global solution to ~\eqref{wm} with data $(u_0, u_1)$.  We have now completed the proof of Theorem \ref{main thm}. We summarize the entire proof below.

\begin{proof}[Proof of Theorem \ref{main thm}]
In Proposition ~\ref{global est} we established the global a priori bounds \eqref{global estimates} for smooth wave maps $(u, \dot{u})$ with initial data $(u_0, u_1)$ that satisfies \eqref{small energy}. Now, given data $(u_0, u_1) \in H^2\times H^1((\R^4, g), TN)$ satisfying ~\eqref{small energy} the above argument concludes the existence of a global wave map $(u, \dot{u}) \in C^0( \R; H^2 \times H^1)$. Proposition ~\ref{global est}, and in particular the global control of the $L^2_tL^8_x$ norm of $du$ allowed us to deduce the global regularity result, Proposition ~\ref{high reg}, which not only drives the existence proof above, but also shows that higher regularity of the data is preserved. Finally, the global control of the $L^2_tL^8_x$ norm of $du$ validates the uniqueness proof in Section ~\ref{uniqueness}.
\end{proof}

\vspace{\baselineskip}

\section{Linear Dispersive Estimates for Wave Equations on a Curved Background}\label{Strichartz Estimates}

In this section we outline the linear dispersive estimates for variable coefficient wave equations established by Metcalfe and Tataru in ~\cite{Met-Tat}.  We review a portion of the argument in ~\cite{Met-Tat} with the necessary extensions needed to prove ~\eqref{Besov Strichartz}. It is suggested that the reader refer to ~\cite{Met-Tat} when reading this section as here we detail only the parts where changes have been made to suit our needs.  In order to facilitate this joint reading we will try to use as much of the same notation as possible. We begin with a brief summary.

We say that $(\rho, p, q)$ is a Strichartz pair if $2\le p \le \infty$, $2\le q< \infty$, and if $(\rho, p, q)$ satisfies the following two conditions

\begin{gather}\label{p,q, rho}
\frac{1}{p} + \frac{d}{q} = \frac{d}{2} - \rho\\ \notag
\\
\frac{1}{p} + \frac{d-1}{2q} \le \frac{d-1}{4}\label{p,q}
\end{gather}
with the exception of the forbidden endpoint $(1,2, \infty)$, if $d=3$.

In ~\cite{Met-Tat}, Metcalfe and Tataru prove global Strichartz estimates for variable coefficient wave equations of the form 

\begin{align}\label{var wave}
Pv & = f\\ \notag
v[0] &= (v_0, v_1)
\end{align}
where $P$ is the second order hyperbolic operator, 
\begin{align}
P(t,x, \partial) = -\p_t^2 + \partial_{\al}( a^{\al \be}(x) \partial_{\be}) + b^{\al}(x) \partial_{\al} + c(x)
\end{align}
In fact, in ~\cite{Met-Tat}  time-dependent coefficients are considered as well, but we will restrict our attention to the time-independent case for our purposes. Here we assume that the matrix $a$ is positive definite and the coefficients $a$, $b$, $c$ satisfy the weak asymptotic flatness conditions

\begin{align}\label{a flat}
\sum_{j\in \Z} \sup_{\abs{x} \sim 2^j} \abs{x}^2 \abs{\p^2 a(x)} + \abs{x}\abs{\p a(x)} + \abs{ a(x) - g_0} \le \ti{\ep}
\end{align}
where $g_0$ denotes the diagonal matrix $\textrm{diag}(1, \dots, 1)$.  And 

\begin{align}\label{b flat}
\sum_{j\in \Z} \sup_{\abs{x} \sim 2^j} \abs{x}^2 \abs{\p b(x)} + \abs{x}\abs{ b(x)}  \le \ti{\ep}
\end{align}

\begin{align}\label{c flat}
\sum_{j\in \Z} \sup_{\abs{x} \sim 2^j} \abs{x}^4 \abs{ c(x)}^2   \le \ti{\ep}
\end{align}
 Given the assumptions in ~\eqref{g}--\eqref{p^2 g}, it is clear that our wave equation for $q$ in ~\eqref{scalar eqn for q} is of this form. Metcalfe and Tataru introduce the following function spaces in order to deduce localized energy estimates and control error terms later on.
 
 Let $S_k$ denote the $k$th Littlewood Paley projection. Set $ A_j= \R \times\{\abs{x} \simeq 2^j\}$ and $A_{<j} =\R \times\{\abs{x} \lesssim 2^{j}\}$. 
  For a function $v$ of frequency $2^k$ define the norm
 
 \begin{align}\label{2^k norm}
 \|v\|_{X_k} := 2^{\frac{k}{2}}\|v\|_{L^2_{t,x}(A_{<-k})} + \sup_{j\ge -k} \|\abs{x}^{-\frac{1}{2}} v\|_{L^2_{t,x}(A_j)}
 \end{align} 
 With this we can define the global norm

\begin{align} \label{local energy norm}
\|v\|^2_{X^s} := \sum_{k\in \Z} 2^{2sk}\|S_k v\|_{X_k}^2
\end{align}
for $-\frac{d+1}{2} < s < \frac{d+1}{2}$. The space $X^s$ is defined to be the completion of all Schwartz functions with respect to the $X^s$ norm defined above. For the dual space $Y^s~=~(X^{-s})^{\prime}$ we have the norm

\begin{align}\label{Y^s}
\|f\|_{Y^s}^2 = \sum_{k\in\Z} 2^{2sk}\|S_k f\|_{X_k^{\prime}}^2
\end{align}
for $-\frac{d+1}{2} < s < \frac{d+1}{2}$. We refer the reader to ~\cite{Met-Tat} for details regarding the structure of these spaces. 

With this setup, Metcalfe and Tataru are able to prove the following results:
\begin{enumerate}
 \item Establish $\dot{H}^{s}$ localized energy estimates for the operator $P$, see ~\cite[Definition 2, Theorem 4 and Corollary 5]{Met-Tat}. 
 \item Construct a global-in-time parametrix, $K$, for the operator $P$, and prove Strichartz estimates for this parametrix.  Error terms are controlled in the localized energy spaces, see ~\cite[Propositions 15-17 and Lemmas 19-21]{Met-Tat}. 
 
\item Combine the localized energy estimates with the Strichartz and error estimates for the parametrix to prove global Strichartz estimates for solutions to ~\eqref{var wave}, see ~\cite[Theorem 6]{Met-Tat}. 
 \end{enumerate}  
\vspace{\baselineskip}

To prove these results, Metcalfe and Tataru are able to make a number of simplifications that allow them to treat the lower order terms in $P$ as small perturbations and work instead with only the principal part of $P$, denoted by $P_a= -\p_t^2 + \p_{\al}a^{\al \be} \p_{\be}$. 

Let $\chi_j$ be smooth spatial Littlewood-Paley multipliers, i.e.

\begin{align*}
1= \sum_{j\in \Z} \chi_{j}(x),   \hspace{.4in} \textrm{supp}(\chi_j) \subset \{ 2^{j-1} \le \abs{x} \le 2^{j+1}\}  
\end{align*}
And set 
\begin{align*}
\chi_{< j}(x) :=  \sum_{k<j} \chi_k(x),  \hspace{.4in} S_j := \sum_{k<j} S_k
\end{align*}
We then define frequency localized coefficients 

\begin{align}\label{a loc}
a^{\al \be}_{(k)}:= g_0^{\al \be} + \sum_{\ell<k-4} (S_{<\ell} \chi_{< k-2\ell}) S_{\ell} a^{\al \be}
\end{align}
corresponding frequency localized operators 

\begin{align}\label{P_{(k)}}
P_{(k)} := -\p_t^2 + \p_{\al}( a^{\al \be}_{(k)} \p_{\be})
\end{align}
used on functions of frequency $k$, and the global operators

\begin{align}\label{tiP}
\ti{P} := \sum_{k\in\Z} P_{(k)} S_k
\end{align}

 In order to prove ~\eqref{Besov Strichartz}, we only need to make a small alteration to the proof of the Strichartz estimates for the parametrix, $K$. At first, the parametrix construction occurs on the level of the frequency localized operator, $P_0$, see \cite[Propositions 15-17]{Met-Tat}.  In particular, they prove the following result.
 
 \begin{prop}[\cite{Met-Tat}, Proposition 17]\label{prop 17}
 Assume that $\ti{\ep}$ is sufficiently small, and assume that $f$ is localized at frequency $0$. Then there is a parametrix $K_0$ for $P_{(0)}$ which has the following properties:
 
\begin{list}{(\roman{parts})}{\usecounter{parts}}

\item (regularity) For any Strichartz pairs $(p_1, q_1)$ respectively $(p_2, q_2)$ with $q_1\le q_2$, we have 
\begin{align}\label{K_0}
\|\p K_{0} f\|_{L_t^{p_1}L_x^{q_1} \cap X_0} \lesssim \|f\|_{L_t^{p_2^{\prime}}L_x^{q_2^{\prime}}}
\end{align}

\item (error estimate) For any Strichartz pair $(p,q)$ we have 
\begin{align}\label{K_0 error}
\|(P_{(0)}K_{0}-1)f\|_{X_0^{\prime}} \lesssim  \|f\|_{L_t^{p^{\prime}}L_x^{q^{\prime}}}
\end{align}

\end{list}
\end{prop}
The next step is to move from these frequency localized parametrices  to a construction of a parametrix for $P_a$, and this is where we make a slight alteration. To begin, Metcalfe and Tataru prove that the operator $P_{(0)}$ in Proposition ~\ref{prop 17} can be replaced with $\ti{P}$, see ~\cite[Lemma 10]{Met-Tat}, on functions localized at frequency $0$. To construct parametrices, $K_j$,  at any frequency $j$, for $\ti{P}$ we rescale, setting

\begin{align}\label{K_j}
K_jf (t,x) = 2^{-2j} K_0 (f_{2^{-j}}) (2^j t, 2^jx)
\end{align}
where $f_{2^{-j}}(t,x) = f(2^{-j}t, 2^{-j} x)$. Next, set 

\begin{align}\label{K}
K := \sum_{j\in \Z} K_j S_j
\end{align}
 
With these definitions it is straightforward to prove the following lemma, which is our altered version of ~\cite[Lemma 19]{Met-Tat}. Recall that the homogeneous Besov norm of a function $\varphi$ is given by

\begin{align*}
\|\varphi\|_{\dot{B}^{s}_{p,q}} = \left( \sum_{j\in Z} (2^{sj}\|S_j\varphi\|_{L^p})^q \right)^{\frac{1}{q}}
\end{align*}
Then we have 

\begin{lem}[Besov space version of Lemma 19 in ~\cite{Met-Tat}] \label{besov K}
The parametrix $K$ has the following properties:

\begin{list}{(\roman{parts})}{\usecounter{parts}}

\item (regularity) For any Strichartz pairs $(\rho_1, p_1, q_1)$, respectively $(\rho_2, p_2, q_2)$ with $q_1\le q_2$ we have 

\begin{align}\label{K Strichartz}
\|\p K f\|_{L_t^{p_1}\dot{B}^{s-\rho_1}_{q_{1},2} \cap X^s} \lesssim \|f\|_{\abs{\partial_x}^{-s-\rho_2}L_t^{p_2^{\prime}}L_x^{q_2^{\prime}}}
\end{align}

\item (error) For any Strichartz pair $(\rho, p, q)$, we have 

\begin{align}\label{K error}
\|(P_aK-1)f\|_{Y^s} \lesssim  \|f\|_{\abs{\partial_x}^{-s-\rho}L_t^{p^{\prime}}L_x^{q^{\prime}}}
\end{align}

\end{list}
\end{lem}

\begin{proof}
We begin by extending the results of Lemma ~\ref{prop 17} to the parametrices $K_j$. Observe that $\p K_j f(t,x) = 2^{-j} \p K_0(f_{2^{-j}})(2^jt, 2^jx)$. Hence, for a function $f$ localized at frequency $j$, we have 

\begin{align*}
\|\p K_j f\|_{L_t^{p_1}L_x^{q_1}} &= 2^{-j}2^{j(-\frac{1}{p_1} - \frac{d}{q_1})} \|\p K_0 (f_{2^{-j}})\|_{L_t^{p_1}L_x^{q_1}}\\
\\
&\lesssim 2^{j(-1 -\frac{1}{p_1} - \frac{d}{q_1})} \|f_{2^{-j}}\|_{L_t^{p_2^{\prime}}L_x^{q_2^{\prime}}}\\
\\
&= 2^{j(-1 -\frac{1}{p_1} - \frac{d}{q_1})} 2^{j(\frac{1}{p_2^{\prime}}+ \frac{d}{q_2^{\prime}})} \|f\|_{L_t^{p_2^{\prime}}L_x^{q_2^{\prime}}}
\end{align*}
Therefore, by ~\eqref{p,q, rho} we obtain 

\begin{align}\label{freq j}
2^{j(\frac{d}{2} +1 - \rho_1)}\|\p K_j f\|_{L_t^{p_1}L_x^{q_1}} \lesssim 2^{j(\frac{d}{2} +1 + \rho_2)}\|f\|_{L_t^{p_2^{\prime}}L_x^{q_2^{\prime}}}
\end{align}
for functions $f$ localized at frequency $j$.

We also need to estimate $\|\p K_j f\|_{X_j}$. Let $f$ again be localized at frequency ~$j$. Observe that 

\begin{align*}
2^{\frac{j}{2}}\| \p K_j f\|_{L_{t,x}^2(A_{<-j})} &= 2^{\frac{j}{2}}\left( \int_{\abs{x}\le 2^{-j}} \abs{\p K_jf(t,x)}^2 \, dx\, dt\right)^{\frac{1}{2}}\\
\\
&= 2^{j(-1-\frac{d}{2})}\|\p K_0 f_{2^{-j}}\|_{L_{t,x}^2(A_{<0})}
\end{align*}
Therefore we can apply Proposition ~\ref{prop 17} to deduce 

\begin{align*}
2^{j(\frac{d}{2} +1)} 2^{\frac{j}{2}}\| \p K_j f\|_{L_{t,x}^2(A_{<-j})} &\lesssim \|f_{2^{-j}}\|_{L_t^{p_2^{\prime}}L_x^{q_2^{\prime}}}\\
\\
&=2^{j(\frac{1}{p_2^{\prime}} + \frac{d}{q_2^{\prime}})}\|f\|_{L_t^{p_2^{\prime}}L_x^{q_2^{\prime}}}\\
\\
&= 2^{j(\frac{d}{2} +1 + \rho_2)}\|f\|_{L_t^{p_2^{\prime}}L_x^{q_2^{\prime}}}
\end{align*}
Similarly one can show for any $k\ge -j$ that

\begin{align*}
2^{j(\frac{d}{2} +1)}\norm{\abs{x}^{-\frac{1}{2}} \p K_j f}_{L_{t,x}^2(A_k)} \lesssim 2^{j(\frac{d}{2} + 1+ \rho_2)}\|f\|_{L_t^{p_2^{\prime}}L_x^{q_2^{\prime}}}
\end{align*}
Hence

\begin{align*}
\|\p K_j f\|_{X_j} \lesssim 2^{j\rho_2} \|f\|_{L_t^{p_2^{\prime}}L_x^{q_2^{\prime}}}
\end{align*}

Putting this all together we obtain the frequency $j$ version of Proposition ~\ref{prop 17} (i):

\begin{align}\label{freq j prop 17}
2^{-j\rho_1} \|\p K_j f\|_{L_t^{p_1}L_x^{q_1}} + \|\p K_j f\|_{X_j} \lesssim 2^{j\rho_2} \|f\|_{L_t^{p_2^{\prime}}L_x^{q_2^{\prime}}}
\end{align}

The next step is to use the Littlewood-Paley theorem to sum up these frequency localized pieces. As a preliminary step we observe that ~\eqref{freq j prop 17} implies that for each~ $s$

\begin{align}\label{freq j est}
2^{2j(s-\rho_1)} \|\p K_j f\|_{L_t^{p_1}L_x^{q_1}}^2 + 2^{2js}\|\p K_j f\|_{X_j}^2 \lesssim 2^{2j(s+\rho_2)} \|f\|_{L_t^{p_2^{\prime}}L_x^{q_2^{\prime}}}^2
\end{align}
Then, we have 
\begin{align}
\|\p K f\|_{L_t^{p_1}\dot{B}^{s-\rho_1}_{q_{1},2}}  &= \norm{ \left( \sum_{j\in\Z} 2^{2j(s-\rho_1)}\norm{S_j\p \sum_{\ell\in\Z} K_{\ell} S_{\ell} f}^2_{L_x^{q_1}} \right)^{\frac{1}{2}}}_{L_t^{p_1}}\\ \notag
\\
&\lesssim \norm{ \left( \sum_{j\in\Z} 2^{2j(s-\rho_1)}\norm{\p  K_j S_j f}^2_{L_x^{q_1}} \right)^{\frac{1}{2}}}_{L_t^{p_1}}\\ \notag
\\
&\lesssim \left( \sum_j 2^{2j(s-\rho_1)}\norm{\p  K_j S_j f}^2_{L_t^{p_1}L_x^{q_1}} \right)^{\frac{1}{2}} \label{Mink 1}\\ \notag
\\
&\lesssim  \left( \sum_j 2^{2j(s+\rho_2)}\norm{S_j f}^2_{L_t^{p_2^{\prime}}L_x^{q_2^{\prime}}} \right)^{\frac{1}{2}} \label{freq j 1}\\ \notag
\\
&\lesssim \norm{ \left( \sum_j 2^{2j(s+\rho_2)}\abs{S_j f}^2\right)^{\frac{1}{2}}}_{L_t^{p_2^{\prime}}L_x^{q_2^{\prime}}} \label{dual Mink 1}\\ \notag
\\
&\lesssim \|f\|_{\abs{\partial_x}^{-s-\rho_2}L_t^{p_2^{\prime}}L_x^{q_2^{\prime}}} \label{L P 1}
\end{align}
Above, ~\eqref{Mink 1}, ~\eqref{freq j 1}, ~\eqref{dual Mink 1} and ~\eqref{L P 1} follow, respectively, from Minkowski's inequality, estimate ~\eqref{freq j est}, the dual estimate to Minkowski, and the Littlewood-Paley theorem.  Finally, we have 
\begin{align}
\|\p K f\|_{X^s} &= \left( \sum_{j\in\Z} 2^{2js}\norm{S_j\p \sum_{\ell\in\Z} K_{\ell} S_{\ell} f}^2_{X_j}\right)^{\frac{1}{2}}\\ \notag
\\
&\lesssim  \left( \sum_{j\in\Z} 2^{2js}\norm{\p K_j S_j f}^2_{X_j}\right)^{\frac{1}{2}}\\ \notag
\\
&\lesssim \left( \sum_{j\in\Z} 2^{2j(s+\rho_2)}\norm{S_j f}^2_{L_t^{p_2^{\prime}}L_x^{q_2^{\prime}}}\right)^{\frac{1}{2}}\\ \notag
\\
&\lesssim \|f\|_{\abs{\partial_x}^{-s-\rho_2}L_t^{p^{\prime}}L_x^{q^{\prime}}} \label{Mink 2 and L P}
\end{align}
where ~\eqref{Mink 2 and L P} follows from the dual to Minkowski's inequality and the Littlewood-Paley theorem. The proof of ~\eqref{K error} follows exactly as in ~\cite{Met-Tat}. 
\end{proof}

We can carry out the rest of the argument exactly as in ~\cite{Met-Tat} except with Lemma ~\ref{besov K} in place of  ~\cite[Lemma $19$]{Met-Tat},  to obtain the following Besov space version of  ~\cite[Theorem $6$]{Met-Tat}. 

\begin{thm}[Besov space version of Theorem $6$ in ~\cite{Met-Tat}] \label{Strichartz}
Let $d\ge 4$. Assume that the coefficients $a^{\al}{\be}$, $b^{\alpha}$, $c$ satisfy ~\eqref{a flat}, ~\eqref{b flat}, and ~\eqref{c flat} with $\ti{\ep}$ sufficiently small. Let $(\rho_1, p_1, q_1)$ and $(\rho_2, p_3, q_2)$ be two Strichartz pairs and assume further that $s=0$ or $s=-1$. Then the solution $v$ to ~\eqref{var wave} satisfies 

\begin{align}\label{Stri}
\|\partial v\|_{L_t^{p_1}\dot{B}^{s-\rho_1} _{q_1, 2}} + \|\p v\|_{X^s} \lesssim \|v[0]\|_{\dot{H}^{s+1} \times \dot{H}^s} + \|f\|_{\abs{\partial_x}^{-s-\rho_2}L_t^{p_2^{\prime}}L_x^{q_2^{\prime}} + Y^{s}}
\end{align}
\end{thm}

To obtain ~\eqref{Besov Strichartz} we set $s=0$, $\rho_1=\frac{5}{6}$, $p=2$, $q=6$, $\rho_2= 0$, $p_2=1$ and $q_2=2$  in ~\eqref{Stri} giving 

\begin{align}
\| v\|_{L^{2}\dot{B}^{\frac{1}{6}} _{6, 2}}  \lesssim \|v[0]\|_{\dot{H}^{1} \times L^2} + \|f\|_{L^{1}L^{2}}
\end{align}
We combine this which the energy estimates which correspond to $s=0$, $\rho_1=0$, $p=\infty$, $q=2$, $\rho_2= 0$, $p_2=1$, $q_2=2$  and $d=4$ in ~\eqref{Stri} giving 

\begin{multline}
\| v\|_{L_t^{2}(\R; (\dot{B}^{\frac{1}{6}} _{6, 2}(\R^4))} + \|\p v\|_{L_t^{\infty}(\R; (L_x^2(\R^4))} \lesssim \\ \|v[0]\|_{\dot{H}^{1} \times L^2(\R^4)} + \|f\|_{L_t^{1}(\R; (L_x^{2}(\R^4))}
\end{multline}
which is exactly ~\eqref{Besov Strichartz}.

\vspace{\baselineskip}

\section{Appendix}\label{Appendix}

\subsection{Sobolev Spaces}\label{sobolev spaces}
We have interchangeably used two different definitions of Sobolev spaces throughout the paper. The difference in the definitions arises from the different ways that we can view maps $f: M \to N$ and their differentials $df: TM \to u^*TN$.  On one hand, we can take the extrinsic viewpoint, where we consider the isometric embedding of  $N\hookrightarrow \R^m$ and view $TN$ as a subspace of $\R^m$. Here we view $f$ as a map  $M \to \R^m$ with values in $N$ and $df: TM \to \R^m$ with values in $TN$. On the other hand, we can view things intrinsically, and exploit the parallelizable structure on $TN$.  We outline these different approaches below, and show that if we take the Coulomb frame on $u^*TN$ these approaches are equivalent for our purposes. Furthermore, we show that if $(M,g)= (\R^4, g)$ with $g$ as in ~\eqref{g}--\eqref{p^k g} then all of the following spaces are equivalent to those that would arise if we had set $M$ to be $\R^4$ with the Euclidean metric. 

\subsubsection{Extrinsic Approach}\label{extrinsic approach}  Taking the extrinsic point of view, we consider maps $f: (M,g) \to (\R^m, \ang{\cdot, \cdot})$. Hence, we can write $f=(f^1, \dots, f^m)$ with the differential $df = (df^1, \dots , df^m)$. For such $f$ and for $1<p<\infty$, we define the norm

\begin{align}\label{sob space}
\|f\|_{W_e^{k,p}}&= \sum_{\ell=0}^{k}  \left( \sum_{a=1}^m \int_M \abs{\n^{\ell} f^a}_g^p \textrm{dvol}_g \right)^{\frac{1}{p}}\\
&=\sum_{{\ell}=0}^{k}  \left( \sum_{a=1}^m \int_M \left( g^{i_1j_1} \cdots g^{i_{\ell} j_{\ell}} (\n^{\ell} f^a)_{i_1, \dots, i_{\ell}} (\n^{\ell} f^a)_{j_1\dots, j_{\ell}}\right)^{\frac{p}{2}} \sqrt{\abs{g}} \,dx \right)^{\frac{1}{p}} \notag
\end{align}
where $\n^{\ell}$ denotes the $\ell$th covariant derivative on $M$ with the convention that $\n^0 f^a =f^a$, see  ~\cite{Heb} . For example, the components in local  coordinates of $\n f^a$ are given by $(\n f^a)_i = (df)_i =\p_i f$ while the components in local coordinates for $\n^2 f^a$ are given by
\begin{align*}
(\n^2 f^a)_{ij}= \p_{ij} f^a -\Gamma_{ij}^k f^a_k
\end{align*}
We define $W^{k,p}_e(M, \R^m)$ to be the completion of  $\{f\in C^{\infty}(M; \R^m): \|f\|_{W_e^{k,p}} <\infty \}$ with respect to the above norm, (the subscript $e$ here stands for extrinsic). We then define $W^{k,p}_e(M, N)$ to be the space of functions ~$\{f \in W^{k,p}_e(M, \R^m): ~f(x) ~\in ~N, \, \textrm{a.e.}\}$.   The homogeneous Sobolev spaces $\dot{W}^{k,p}_e(M; N)$ are defined similarly. 

\begin{rem}The one drawback with this definition is that $C^{\infty}(M; N)$ may not be dense in $W^{1,p}(M; N)$ for $p<\textrm{dim}\, M$, for a generic compact manifold $N$. For example, in ~\cite{Sch-Uhl 2}, Schoen and Uhlenbeck show that $f(x) = \frac{x}{\abs{x}} \in H^1(B^3; S^2)$ cannot be approximated by $C^{\infty}$ maps from $B^3 \to S^2$ in $H^1(B^3; S^2)$, see \cite{Lin-Wang} for a proof. This poses a potential difficulty for us as we required the density of $C^{\infty}(M, TN)$ in $H^1(M; TN)$ in order to  approximate the data $(u_0, u_1) \in H^2\times H^1 (M; TN)$ by smooth functions in our existence argument. Thankfully, this difficulty can be avoided using the equivalence of the extrinsic and intrinsic definitions  of Sobolev spaces which will be argued below.
\end{rem}

With $(M,g) = (\R^4, g)$, with $g$ as in ~\eqref{g}--\eqref{p^k g}, and $\ep$ small enough, we can show that these ``covariant" Sobolev Spaces $W_e^{k,p}((\R^4, g); N)$ are equivalent to the ``flat" Sobolev spaces $W_e^{k,p}((\R^4,\ang{\cdot, \cdot}); N)$. 

\begin{lem}\label{sobolev} Let  $(M,g) = (\R^4, g)$ with $g$ as in ~\eqref{g}--\eqref{p^k g} and let $1<p<\infty $. Then $\dot{W}^{k,p}((\R^4, g))$ is equivalent to $\dot{W}^{k,p}(\R^4, g_0)$ where $g_0$ is the Euclidean metric on $\R^4$. In particular, if $f:\R^4 \to \R^m$ then for every $k\in\N$ we have 
\begin{align}\label{sob equiv}
\|\p^kf\|_{L^p(\R^4)} \simeq \|\n^k f\|_{L^p(\R^4,g)}
\end{align}
\end{lem}

\begin{proof} 
As the above norms are defined component-wise for $f=(f^1,\dots, f^m)$, it enough to prove the statement for functions $f:(M,g) \to \R$ instead of for maps $f:(M,g) \to \R^m$ with values in $N$. We also will only prove this lemma in detail for a few easy cases, namely for $k=0, 1$ and for $k=2, p=2$. These, in fact, include all the cases that we need. The other cases follow by similar arguments. 

By ~\eqref{g} it is clear that $\ds{\sqrt{\abs{g(x)}}}$ is a bounded function on $\R^4$.  Hence, for every $k$ we have 

\begin{align*}
\int_{\R^4} \abs{\n^k f}_g^p \sqrt{\abs{g}} \,dx \simeq \int_{\R^4} \abs{\n^k f}_g^p \,dx
\end{align*}
This proves the lemma for $k=0$. In local coordinates we have, for $k=1$, that $(\n f)_i :=(df)_i=  \p_i f$  and $\abs{\p f}_g^2 = g^{ij} \p_i f\p_j f$. Letting $g_0$ denote the Euclidean metric we have, for $p$ even, that 

\begin{align*}
\abs{\n f}_g^p  -   \abs{\p f}^p  &= (g^{ij} \p_i f\p_j f)^{\frac{p}{2}} -  (g_0^{ij}\p_i f\p_j f)^{\frac{p}{2}}\\
&= (g^{ij}-g_0^{ij})(\p_i f\p_j f) \sum_{\ell=1}^{\frac{p}{2}} (g^{ab}\p_a f\p_b f)^{\frac{p}{2}-\ell}(g_0^{cd}\p_c f\p_d f)^{\ell-1}
\end{align*}
Hence by ~\eqref{g} we have 

\begin{align*}
\abs{\|\n f\|_{L^p(\R^4,g)}^p - \|\p f\|_{L^p(\R^4)}^p } &\lesssim \ep \int_{\R^4} \abs{\p f}^p \,dx
\end{align*}
For $p$ odd we interpolate. This proves the case $k=1$.  For $k=2, p=2$ we have in local coordinates that 

\begin{align*}
(\n^2 f)_{ij} = \p_{ij}f - \Gamma_{ij}^{\ell} \p_{\ell} f
\end{align*}
where here $\Gamma^{l}_{ij}$ are the Christoffel symbols for $(\R^4,g)$. We also have
\begin{align*}
\abs{\n^2 f}^2 = g^{ik}g^{j\ell}(\n^2 f)_{ij}(\n^2 f)_{k\ell}
\end{align*}
Hence, using ~\eqref{g}-\eqref{p g}  and the Sobolev embedding  we have 

\begin{align*}
&\abs{\|\n^2 f\|_{L^2(\R^4,g)}^2 - \|\p^2 f\|_{L^2(\R^4)}^2 }\\
 &\simeq\abs{ \int_{\R^4} g^{ik}g^{j\ell}(\p_{ij}f -\Gamma^{a}_{ij} \p_a f) (\p_{k\ell}f -\Gamma^b_{k\ell} \p_b f) \, - g_0^{ik}g_0^{j\ell} (\p_{ij}f) (\p_{k\ell} f )\,dx}\\
\\
&\lesssim \abs{\int_{\R^4}( g^{ik}g^{j\ell}-g_0^{ik}g_0^{j\ell})(\p_{ij}f) (\p_{k\ell} f) \,dx}+ 2\abs{ \int_{\R^4} g^{ik}g^{j\ell}\, \p_{ij}f\,  \Gamma_{k\ell}^a\, \p_a f\,dx} \\
&\quad+ \abs{\int_{\R^4} g^{ik}g^{j\ell}\,  \Gamma_{ij}^a\, \p_a f\,\Gamma_{k\ell}^b\, \p_b f\,dx}\\
 \\
 &\lesssim \ep^2 \| \p^2 f\|_{L^2(\R^4)}^2 + \|\p^2 f\|_{L^2(\R^4)} \| \Gamma\|_{L^4(\R^4)} \|\p f\|_{L^4(\R^4)} + \|\p f\|_{L^4(\R^4)}^2 \|\Gamma\|^2_{L^4(\R^4)}
\end{align*}
Now, recall that $\Gamma_{ij}^a= \frac{1}{2}g^{ab}(\p_i g_{bj} + \p_j g_{ib} - \p_b g_{ij})$. Hence by ~\eqref{p g}, we have $\|\Gamma\|_{L^4(\R^4)}\lesssim \ep$. Using the Sobolev embedding 
$\dot{H}^{1}(\R^4) \hookrightarrow L^4(\R^4)$ and the above inequalities we have 

\begin{align*}
\abs{\|\n^2 f\|_{L^2(\R^4,g)}^2 - \|\p^2 f\|_{L^2(\R^4)}^2 } \lesssim \ep \|\p^2 f\|_{L^2(\R^4)}^2
\end{align*}
proving ~\eqref{sob equiv} in the case $k=2$, $p=2$. 
\end{proof}

\subsubsection{Intrinsic Approach}\label{intrinsic approach}
Next, we use the parallelizable structure on $TN$ to define ``intrinsic" Sobolev spaces for maps $\psi: TM \to u^*TN$.    

Let $\ti{e} = (\ti{e}_1, \dots, \ti{e}_n)$ be a global orthonormal frame on $TN$ and let $\bar{e} = (\bar{e}_1, \dots, \bar{e}^n)$ be the induced orthonormal frame on $u^*TN$ obtained via pullback. Now, let $\psi: TM \to u^*TN$ be a smooth map, i.e., $\psi$ is a $u^*TN$ valued $1$-form on $M$. Then $\psi$ can  be written in terms of the orthonormal frame $\bar{e}$ on $u^*TN$. The components of $\psi$ in the frame $\bar{e}$ are then given by $\psi^a =\ang{ \psi, \bar{e}_a}_{u^*h}$ and each of these can be viewed as a $1$-form on $M$, i.e., a section of $T^*M$, and can be written in local coordinates as $\psi^a = \psi^a_{\al} dx^{\al}$. 

One way to define the Sobolev norms of  $\psi$ is to ignore the covariant structure on $u^*TN$ and say that $\psi \in \dot{W}_i^{k,p}(M;N)$, (the index $i$ here stands for intrinsic), if all of the components, $\psi^a$, are in $\dot{W}^{k,p}(M; \R)$. And we define

\begin{align}\label{component Sobolev} 
&\|\psi\|_{\dot{W}^{k,p}_i(M;N)}^p := \sum_{a=1}^n \|\psi^a\|^p_{\dot{W}^{k,p}(M)} = \sum_{a=1}^n \int_M \abs{\n^{k}\psi^a}_g^p \textrm{dvol}_g \\
&= \sum_{a=1}^n \int_M \left(g^{i_1 j_1}\cdots g^{i_{k+1},j_{k+1}}(\n^k \psi)^a_{i_1,\dots, i_{k+1}} (\n^k \psi)^a_{j_1,\dots, j_{k+1}}\right)^{\frac{p}{2}} \, \sqrt{\abs{g}} \, dx \notag
\end{align}
where $\n^k$ denotes the $k$th covariant derivative on $M$. By the same argument as above, we can show that in our case, with  $(M,g)=(\R^4, g)$ and $g$ as in ~\eqref{g}--\eqref{p^k g}, these spaces are equivalent to the case where we have the Euclidean metric on $\R^4$, that is, there exist constants $c, C$ such that 
\begin{align}\label{intrinsic equiv}
\|\p^k \psi^a\|_{L^p(\R^4)} \simeq \|\n^k \psi^a\|_{L^p(M)}  
\end{align}

The one glaring issue here, is that this construction will depend, in general, on the choice of frame $\bar{e}$. We can avoid this confusion though in the case where the frame $e$ is the Coulomb frame as in this case the intrinsic norms are equivalent to their extrinsic counterparts in the cases we will be interested in. This issue was addressed in Section \ref{Equivalence of Norms}.

\subsection{Density of $C^{\infty} \times C^{\infty}(M; TN)$  in $H^{2} \times H^{1}(M; TN)$}\label{density}

We set $(M,g) = (\R^4, g)$ with $g$ as in ~\eqref{g}-\eqref{p^k g}. In the existence argument for wave maps we claimed the existence of a sequence of smooth data $(u_0^k, u_1^k)  \to (u_0, u_1)$  in $H^{2} \times H^{1}(M; TN)$. Here we show that such a sequence does, in fact,  exist. That is, we show that  $C^{\infty} \times C^{\infty}(M; TN)$ is dense in  in $H^{2} \times H^{1}(M; TN)$. 

First, observe that $C^{\infty}(M;N)$ is dense in $H_e^2(M; N)$, see   ~\cite[Lemma A.$12$]{Brez-Nir}.  Hence we can find a sequence of smooth maps $u^k_0$ such that $u^k_0 \to u_0$ in $H^{2}(M;N)$.    

Finding a sequence of smooth maps $u_1^k :M \to TN$ such that $u_1^k(x) \in T_{u_0^k(x)}N$ approximating $u_1$ in $H^{1}(M; TN)$ is not as straightforward as we do not know a priori that $C^{\infty}(M; TN) $ is dense in $H_e^{1}(M; TN)$. However, we can use the equivalence of the norms $H_e^{1}(M; TN)$ and $H_i^{1}(M; TN)$ proved in the previous section to get around this issue. 

Let $e$ denote the Coulomb frame on $u_0^*TN$. Since $u_1$ is a section of $u_0^*TN$,  we can find 
one-forms $q_1^a$ over $M$ so that $u_1 = q_1^a e_a$. By the equivalence of the norms $H_e^{1}(M; TN)$ and $H_i^{1}(M; TN)$, we see that $u_1 \in H_e^{1}(M; TN)$ if and only if  $q_1^a \in H^1(TM; \R)$. Since $C^{\infty}$ is dense in $H^1(TM; \R)\simeq H^1(\R^4;\R)$ we can find smooth $(q_1^a)^k$ such that $(q_1^a)^k \to q^a_1$ in $H^1(TM; \R)$.  Now, for each smooth map $u_0^k:M \to N$ we can find the associated Coulomb frame $e^k =(e_1^k, \dots, e_n^k)$.  We then define smooth sections $u_1^k :=  (q_1^a)^k e_a^k$ and by the equivalence of norms explained in Section \ref{Equivalence of Norms} we have $u_1^k \to u_1$ in $H^1_e(M; TN)$ as desired.

\subsection{Lorentz Spaces}\label{Lorentz Spaces}
 To prove the pointwise estimates for the connection form $A$ associated to the Coulomb gauge we need a few general facts about Lorentz spaces. We review these facts below.
$L^{p,r}(\R^d)$ functions are measured with the norm

\begin{align*}
\|f\|_{L^{p,r}} =  \left( \int_0^{\infty} t^{\frac{r}{p}} f^*(t)^r\,  \frac{dt}{t} \right)^{\frac{1}{r}}
\end{align*}
for $0<r< \infty$.  If  $r=\infty$, then

\begin{align*}
\|f\|_{L^{p, \infty}} = \sup_{t>0} t^{\frac{1}{p}} f^*(t)
\end{align*}
 where above we have

\begin{align*}
f^*(t) &= \inf \{\al: d_f(\al) \le t\} \\
\\
d_f(\al) &= \textrm{meas}\{ x: \abs{f(x)} > \al\}
\end{align*}
 A consequence of real interpolation theory is that Lorentz spaces can also be characterized as the interpolation spaces given by
 
 \begin{align}\label{interp}
 L^{p,r}(\R^d) = (L^{p_0}, L^{p_1})_{\theta ,r}
 \end{align}
 where $1\le p_0 < p_1\le \infty$, $p_0< r \le \infty$ and $\frac{1}{p} = \frac{1-\theta}{p_0} + \frac{\theta}{p_1}$. We refer the reader to \cite[Chapter 5.2]{Ber-Lof} for more details. 
 
 Note that the  $L^{p,\infty}$ norm is the same as the weak-$L^p$ norm. Below we record some general properties of Lorentz spaces that were needed in the proof of Proposition ~\ref{main elliptic estimates}. We refer the reader to ~\cite{Gra}, ~\cite{Ber-Lof}, ~\cite{ON}, and ~\cite{Tar}  for more details.
  
\begin{lem}\label{Lor 1} Suppose that $0<p\le \infty$ and $0<r<s \le\infty$. Then  
\begin{list}{(\roman{parts})}{\usecounter{parts}}
\item $L^{p,p} =L^p$
\\
\item If $r<s$ then $L^{p,r} \subset L^{p,s}$
\\
\item If $h: \R^d \to \R$  is defined by $\ds{h(x) = \frac{1}{\abs{x}^{\al}}}$, then $h \in L^{\frac{d}{\al}, \infty}$.
\end{list}
\end{lem}
 The proof of Lemma ~\ref{Lor 1} follows easily from the definitions and can be found for example in  ~\cite[Chapter $1.4.2$]{Gra}.  We also needed the Lorentz space versions of H$\ddot{\textrm{o}}$lder's inequality and Young's inequality and the following duality statement.

\begin{lem}\label{Lor 2}  Suppose that $f\in L^{p_1, r_1}$ and $g\in L^{p_2, r_2}$ where $1\le p_1, p_2 <\infty$ and $1\le r_1, r_2 \le \infty$. Then, 
\begin{list}{(\roman{parts})}{\usecounter{parts}}
\item\label{Holder} $\ds{\|fg\|_{L^{p,r}} \lesssim \|f\|_{L^{p_1, r_1}} \|g\|_{L^{p_2, r_2}}}$ if $\frac{1}{p} = \frac{1}{p_1}+ \frac{1}{p_2}$ and $\frac{1}{r} = \frac{1}{r_1}+ \frac{1}{r_2}$
\\
\item \label{Young} $\ds{\|f\ast g\|_{L^{p,r}} \lesssim \|f\|_{L^{p_1, r_1}} \|g\|_{L^{p_2, r_2}}}$ if $0<\frac{1}{p} = \frac{1}{p_1}+ \frac{1}{p_2}-1$ and\\ $\frac{1}{r} = \frac{1}{r_1}+ \frac{1}{r_2}$
\\
\item $(L^{p,r})^{\prime} = L^{p_1, r_1}$ for $1<p<\infty$, $1<r<\infty$ and $(L^{p,1})^{\prime} = L^{p_1, \infty}$ for $1<p<\infty$, where $\frac{1}{p} +\frac{1}{p_1} =1$ and $\frac{1}{r}+ \frac{1}{r_1} =1$
\end{list}

\end{lem}
 To prove ~$(i)$ above observe that $(fg)^*(t) \le f^*(\frac{t}{2}) g^*(\frac{t}{2})$, see ~\cite[Proposition $1.4.5$]{Gra}. Then apply H$\ddot{\textrm{o}}$lder's inequality. We refer the reader to ~\cite{ON} for the proof of ~$(ii)$ above. And ~~$(iii)$ is proved in ~\cite[Theorem $1.4.17$]{Gra}.

We also require Sobolev embedding theorems for Lorentz spaces which can be obtained via real interpolation.  A detailed proof can be found in ~\cite[Chapter 32]{Tar}. 

\begin{lem}[Sobolev embedding for Lorentz spaces] \label{Sob embedding Lorentz}If  \,$\ds{0<s<\frac{d}{q}}$ and $\ds{\frac{1}{p} = \frac{1}{q} - \frac{s}{d}}$ then $\ds{\dot{W}^{s,q}(\R^d) \hookrightarrow L^{p,q}(\R^d)}$ and $\ds{\dot{B}^{s}_{q, r}(\R^d)\hookrightarrow L^{p,r}(\R^d)}$.
\end{lem}
 
To give an idea of why Lemma \ref{Sob embedding Lorentz} is true, we demonstrate a special case, namely that \begin{align}
\dot{H}^{s}(\R^d) \hookrightarrow L^{p,2}(\R^d)\label{special}
\end{align}
for $\frac{1}{p} = \frac{1}{2} - \frac{s}{d}$.  Observe that this is a strengthening of the standard Sobolev inequality which says that $\dot{H}^s(\R^d) \hookrightarrow L^p(\R^d)$ for $\frac{1}{p} = \frac{1}{2} - \frac{s}{d}$ since ~$L^{p,2}(\R^d) \hookrightarrow ~L^p(\R^d)$. 
The proof of \eqref{special} relies on Plancherel's theorem and real interpolation. Let $\FF$ denote the Fourier transform. Let $f \in \dot{H}^s(\R^d)$, which means that $\abs{\xi}^s \FF f \in L^2(\R^d)$. Also note that if 
$0<s< \frac{d}{2}$ then 
\begin{align*}
 \abs{\xi}^{-s} \in L^{\frac{d}{s}, \infty}(\R^d)
\end{align*}
Hence, by H\"older's inequality for Lorentz spaces
\begin{align*}
\|\FF f\|_{L^{\gamma, 2}} = \| \abs{\xi}^s \FF f \abs{\xi}^{-s}\|_{L^{\ga,2}} \lesssim \| \abs{\xi}^s \FF f\|_{L^{2,2}} \|\abs{\xi}^{-s}\|_{L^{\frac{d}{s}, \infty}} < \infty
\end{align*}
for $\frac{1}{\ga} = \frac{1}{2} + \frac{s}{d}$.  Now recall that $\FF^{-1} : L^1 \to L^{\infty}$ and $\FF^{-1} : L^2 \to L^2$. Therefore, by real interpolation
\begin{align*}
\FF^{-1} : (L^1, L^2)_{\theta, 2} \to (L^{\infty}, L^2)_{\theta, 2}
\end{align*}
which, by \eqref{interp} is exactly the statement that 
\begin{align*}
\FF^{-1} : L^{\al, 2}(\R^d) \to L^{\be,2}(\R^d)
\end{align*}
where $\frac{1}{\al} = 1 - \frac{\theta}{2}$ and $\frac{1}{\be} = \frac{\theta}{2}$ and we notice that $\frac{1}{\al}+ \frac{1}{\be} =1$. Hence, with $\frac{1}{\ga} = \frac{1}{2} + \frac{s}{d}$ we have that $\FF f \in L^{\ga, 2}(\R^d)$ which implies that $f \in L^{\ga^{\prime}, 2}(\R^d)$ where $\frac{1}{\ga^{\prime}} =\frac{1}{2} - \frac{s}{d}$ which is exactly \eqref{special}. 

 The $L^p$ and Besov space versions of this statement are slightly more complicated to prove as they require additional facts from real interpolation theory and we refer the reader to \cite{Tar} for a detailed proof. 

We also need the following version of the Calderon-Zygmund theorem for Lorentz spaces. 
\begin{thm}[Calderon-Zygmund theorem for Lorentz spaces]\label{CZ} Let $T$ be a Calderon-Zygmund operator. Then $T: L^{p,r} \to L^{p,r}$ for $1<p<\infty$ and $1\le r\le \infty$,
\begin{align*}
\|T f\|_{L^{p,r}} \lesssim \|f\|_{L^{p,r}}
\end{align*}
where the constant above does not depend on $r$. 
\end{thm}

This extension of the Calderon-Zygmund theorem is an easy consequence of the $L^p$ version given the following interpolation theorem of Calderon, see ~\cite[Theorem 5.3.4]{Ber-Lof}. 
\begin{thm}[Calderon's interpolation theorem]\label{CZ I}Let T be a linear operator and suppose that 

\begin{align*}
T: L^{p_1, \rho} \to L^{q_1, \infty}\\
T: L^{p_2, \rho} \to L^{q_2, \infty}
\end{align*}
where $\rho>0$. Then,
\begin{align*}
T: L^{p, r} \to L^{q, s}
\end{align*}
as long as $0<r\le s\le \infty$, $p_1 \neq p_2$, $q_1\neq q_2$, $\ds{\frac{1}{p} = \frac{(1-\theta)}{p_1} + \frac{\theta}{p_2}}$, and  $\ds{\frac{1}{q} = \frac{(1-\theta)}{q_1} + \frac{\theta}{q_2}}$ for $\theta \in (0,1)$.
\end{thm}
 
 \vspace{.2in}
 
 \begin{proof}[Proof of Theorem ~\ref{CZ}]Let $T$ be a Calderon-Zygmund operator. To prove that  $T: L^{p,q}\to L^{p,q}$, find $p_1, p_2, \theta$ so that $1<p_1 < p < p_2<\infty$  and  $\ds{\frac{1}{p} = \frac{(1-\theta)}{p_1} + \frac{\theta}{p_2}}$.  Then we have 
$T: L^{p_1, p_1} \to L^{p_1, \infty}$ and $T:L^{p_2, p_1} \to L^{p_2, \infty}$ since

\begin{align*}
&\|Tf\|_{L^{p_1, \infty}} \lesssim \|Tf\|_{L^{p_1, p_1}} = \|Tf\|_{L^{p_1}} \lesssim \|f\|_{L^{p_1}} = \|f\|_{L^{p_1, p_1}}\\
\\
&\|Tf\|_{L^{p_2, \infty}} \lesssim \|Tf\|_{L^{p_2, p_2}} = \|Tf\|_{L^{p_2}} \lesssim \|f\|_{L^{p_2}} = \|f\|_{L^{p_2, p_2}}\lesssim \|f\|_{L^{p_2, p_1}}
\end{align*}
Therefore, by Theorem ~\ref{CZ I}, we have $T: L^{p,q} \to L^{p,q}$ for every $q>0$.
\end{proof}

\end{document}